%% file: 1-file-CLT.tex
\documentclass[a4paper, 10pt]{preprint}
\usepackage[left=1.5in, right=1.5in, bottom=1.5in]{geometry}
\usepackage{comment}
\usepackage{mhenvs}
\usepackage{mhsymb}
\usepackage[full]{textcomp}
\usepackage[osf]{newtxtext}
\usepackage{amssymb}
\usepackage{amsmath}
\usepackage{mathtools}
\usepackage{mhequ}
\usepackage{booktabs}
\usepackage{tikz}
\usepackage{mathrsfs}
\usepackage{longtable}
\usepackage{microtype} 
\usepackage{wasysym}
\usepackage{centernot}
\usepackage{enumitem}
\usepackage[dvipsnames]{xcolor}
\usepackage{hyperref}
\usepackage{crossreftools}
\usepackage{physics}
\usepackage{orcidlink}

\def\f{\frac}

\def\cY{\mathcal Y}

\def\1{\mathbf{1}}

\def\L{{\mathbb L}}

\def\R{\mathbb{R}}
\def\E{{\mathbb E}}

\usetikzlibrary{arrows}
\tikzset{
	partition/.style={gray!50, line width=0.8cm,line cap=round,line join=round},
    pairing/.style={thick,draw=black,*-*,shorten >=-2.9pt,shorten <=-2.9pt},
  }

\input{prelude.tex}

\begin{document}
\title{Fluctuations from a random fractional averaging limit}

\author{\normalsize Xue-Mei Li\inst{1,2}\orcidlink{0000-0003-1211-0250},
  Colin Piernot\inst2\orcidlink{0009-0007-4448-5328}, Szymon Sobczak\inst2 \orcidlink{0009-0006-7106-250X}, and
  Kexing Ying\inst2\orcidlink{0000-0002-8292-4746}}

\institute{Department of Mathematics, Imperial College London, U.K.
  \and Institute of Mathematics, EPFL, Switzerland}

\maketitle

\begin{abstract}
  We consider a multiscale system of stochastic differential equations whose slow component is driven 
  by a fractional Brownian motion with Hurst parameter $H>\f 12$. Under ergodic assumptions on the 
  fast component ensuring the applicability of the fractional averaging and fractional homogenization theorems 
  \cite{Hairer-Li:20, Hairer-Li:22}, we establish a fluctuation result. Under fractional scaling, 
  the fluctuation of the slow motion from its limit converges to the solution of a stochastic differential equation 
  driven by a fractional Brownian motion and an additional space–time Gaussian field. 
  Since the averaging principle and the fractional homogenization hold in different modes of convergence,
  obtaining the required joint convergence of the two drivers is a delicate matter and constitutes the core of the proof.
  The fluctuation theorem follows from continuity properties of nonlinear Young integrals, together with a residue-type stability estimate
  for linear Young differential equations.

  \noindent
  {\scriptsize {\it Keywords:} fluctuation, fractional averaging, multiscale,
    fractional Brownian motion, nonlinear Young integral}\\
  {\scriptsize \textit{MSC Subject classification:  60F99, 60H10, 60L20,  60G22}}
\end{abstract}

\setcounter{tocdepth}{2}
\tableofcontents
\newpage

\section{Introduction}

We consider the stochastic differential equation
\begin{equation}\label{eq:x-epsilon}
  \dd x_t^\epsilon = f(x_t^\epsilon, y_t^\epsilon) \dd B^H_t + g(x_t^\epsilon, y_t^\epsilon) \dd t,
  \qquad  x^\epsilon_0 = x_0\in \mathbb{R}^d
\end{equation}
where $y_t^\epsilon=y_{ t/\epsilon}$ with $(y_t)_{t \ge 0}$ being a stationary Markov process on a 
separable metric space $\cY$ with Markov semigroup $(P_t)$ and generator $\CL$. Moreover, the process $(B^H_t)_{t \ge 0}$ 
denotes a fractional Brownian motion in $\R^m$ with Hurst parameter $H \in (\frac{1}{2}, 1)$, 
independent of $y$. Consequently, this system exhibits both timescale separation and 
non-Markovian stochasticity. We assume that  $f(x,y)\in {\mathbb L}(\R^m, \R^d)$ (is a linear map from 
$\R^m$ to $\R^d$) and $g(x,y)\in \R^d$, and that they are sufficiently regular, as specified in
Assumption~\ref{ass:C3} below. The precise interpretation of the equation, due to \cite{Hairer-Li:20}, is
recalled in Section~\ref{sec:z-eps-YDE}.

It was shown in \cite{Hairer-Li:20} that a fractional averaging principle holds for the system 
\eqref{eq:x-epsilon}. As the scale-separation parameter~$\epsilon\to 0$, the fast process $y$ 
equilibrates on the timescale of $O(1)$ relevant to the slow variable $x^\epsilon$.
Consequently, the influence of $y$ on  $x^\epsilon$ averages out with respect to the invariant 
measure of $y$, leading to an autonomous effective equation whose solution approximates~$x^\epsilon$.
In particular, $x^\epsilon$ converges in probability in $C^\alpha$, for any $\alpha<H$, with 
respect to the joint randomness of $y$ and $B_t^H$, to the solution of an equation of the 
same form:
\begin{equation}\label{eq:bar-x}
  \dd \bar x_t = \bar f(\bar x_t) \dd B^H_t+\bar g (\bar x_t)\dd t, \quad \bar x_0 = x_0.
\end{equation}
The effective coefficients $\bar f, \bar g$ are obtained through naïve averaging. For any function 
$\phi:\mathbb{R}^d\times \cY\to \mathbb{R}^d$, we set
$$\bar \phi(x) = \int_\cY \phi(x,y)\dd \mu(y)$$
where $\mu$ denotes the invariant measure of the fast process $y$. In particular, one has the 
following limit in probability in $C^\alpha$
$$\int_{s}^{t} f(x, y^\epsilon_r) \dd B_r^H \stackrel{(\mathbb{P})}{\longrightarrow} \bar f(x) B_{s,t}^H,$$
where we use the shorthand $x_{s,t}=x_t-x_s$ for time increments.

Furthermore, in the case where the averaging limit is trivial, i.e. $\bar f=0$, a fractional 
generating diffusion limit theorem was established~\cite{Hairer-Li:22} for \eqref{eq:x-epsilon}, 
and one has in particular the following convergence in distribution
$$\epsilon^{\f 12-H} \int_s^t f(x, y^\epsilon_r) \dd B_r^H  \stackrel{(d)}{\longrightarrow}  V_{s,t}(x)$$
where $V$ is a space-time Gaussian field with covariance given below in Equation \eqref{eq:V-covariance}.

In the case when $H=\f 12$, this reduces to the classical stochastic averaging principle, while for $H=1$ it 
yields a functional central limit theorem. Indeed, for a zero-mean function $g: \cY\to \R$, one expects
$$\f 1{\sqrt{\epsilon}}\int_s^t g(y^\epsilon_r)
  \dd r  \stackrel{(d)}{\longrightarrow} \sqrt{2\int_0^\infty \<P_sg, g\>_{L^2(\mu)} \dd s}  \; W_{s,t},$$
where $W$ is a standard Wiener process. Functional central limit theorems (extensions of
Donsker’s invariance principle) have been established and developed extensively in the probability 
literature, including \cite{Donsker51, Freidlin-Wentzell, Helland82, Masi-Ferrari-Goldstein-Wick84, Kipnis-Varadhan86}. 
See also the more recent related works \cite{Landim03,Li:08, Kmorowski-Landim-Olla12,  
Gautam-Komorowski-Novikov-Rhyzhik14, Li:16, Li:ODE, Wang-Chao-Duan, Roeckner-Xie-Yang:23}. For reduction 
to slow-fast systems, see \cite{Li:cons, Ye-Yang-Maggioni:24}. Slow-fast systems of equations with 
fractional Brownian motion as drivers have been studied recently; see \cite{Gehringer-Li21,Gehringer-Li22,
Gehringer-Li-Sieber:22, Eichinger-Kuehn-Neamu20, Blomker-Neamtu22, Alonso-Martin-Boedihardjo-Papavasiliou25, 
Berglund-Blessing25}, however the fractional averaging theorem and the fractional homogenization 
results in \cite{Hairer-Li:20, Hairer-Li:22} are the first of their kind. In particular, the 
effective limits differ from those in the earlier literature, and the associated scaling exponents 
also differ. We further note a recent surge of activity in multiscale analysis with memory
\cite{Nourdin-Peccati-Reinert:10,Gu-Bal:12,Friz-Gassaiat-Lyons:15,Solesne-Siragan-Konstantinos21,Solesne-Dang-Spiliopoulos24,Roeckner-Xie21,Han-Xu-Pei-Wu, 
Pei-Inahama-Xu21,Wang-Xu-Pei:23, Jaramillo-Nopurdin-Nualart-Peccati:23,Djurdjevac-Kremp-Perkowski25,Gailus-Gasteratos25, Friz-Salkeld-Wagenhofer25, Chekroun-Liu-McWilliams25,bianchi-Bonaccorsi-Canadas-Friesen,Brehier-Faye:25,Li-Gao-Qu:25,Xu-Lian-Wu:23}. For brevity, we omit references to two-scale SDEs driven by fractional 
Brownian motion that reduce to the classical averaging principle, either because the multiplicative 
coefficient in front of the fractional Brownian motion does not depend on one of the variables, 
or because the equation can be transformed into such a form.

The natural next step is to quantify the fluctuations in the case where the averaging is non-trivial (and moreover random), i.e. when $\bar f \neq 0$,
hence bringing together the two aforementioned results. This is the main result of this paper, which we now state informally. 

Guided by the fractional homogenization theorem,  we scale the difference
$x_t^\epsilon-\bar x_t$ by $\epsilon^{\f 12 -H}$. Under this scaling, the fluctuations of the
drift $g$ around its average $\bar g$ vanish in the limit since they occur with CLT scaling
$\epsilon^{-\f 12}$, which vanishes under the scaling $\epsilon^{\f 12 -H}$ of
the fluctuations created by the fractional noise (cf.\ Lemma~\ref{lem:G-eps-vanish} below).
Nonetheless, the drift leaves a trace in the limiting dynamics. As we will see in the limit
equation \eqref{eq:z-lim-eq}, multiplicative terms of the form $D\bar f(\bar x_s) z_s \dd B^H_s$
and $D \bar g(\bar x_s) z_s \dd s$ both appear.

Moreover, we will see below that, in general, the random fluctuation is not a Markov process and that the limit involves both Itô integration and Young integration with respect to a fractional Brownian 
motion. This stands in contrast with the fractional homogenization limit in \cite{Hairer-Li:22}, 
where the limit is Markovian and the long-range memory is lost.

The precise statement of our main result requires some ergodicity and regularity conditions. Namely,
Assumptions~\ref{ass:ergodic} and \ref{ass:C3} will be stated in Section~\ref{sec:assumptions} below
and are adapted from \cite{Hairer-Li:22}. We first introduce some required definitions. Let
\begin{equ}\label{e:v-rp}
  \begin{aligned}
    V_{s,t}^\epsilon := x \mapsto V^{\epsilon}_{s,t}(x)
      = \epsilon^{\f 12-H}\int_s^t \left[f(x, y^\epsilon_r) - \bar f(x)\right] \dd B^H_r, \quad x\in \mathbb{R}^d,
  \end{aligned}\end{equ}
for which the integral can be understood as a Wiener integral and in particular, is of regularity 
\(\beta\) for any \(\beta < \frac{1}{2}\) (cf. \cite[Proposition 3.1]{Hairer-Li:22} or Proposition~\ref{prop:V-eps-moments}).
Moreover, let $V$ be a space-time Gaussian field independent of $B^H$ with the covariance
\begin{equ}[eq:V-covariance]
  \mathbb{E}[V_{s, t}(x) \otimes V_{u, v}(x')]
  = |[s, t] \cap [u, v]| (\Sigma(x, x') + \Sigma(x', x)^{\top}),
\end{equ}
where $\Sigma$ is given by 
\begin{equ}[e:defSigma]
  \Sigma(x, x')  = {1\over 2}\Gamma(2H+1)\sum_{k=1}^m\int (f_k(x,y)-\bar f_k(x)) \otimes \big(\CL^{1-2H} (f_k-\bar f_k)\big)(x',y)\mu(\dd y),
\end{equ}
where \(f_k\) denotes the $k$-th component of $f$ and $\bar f_k$ denotes the $k$-th component of $\bar f$. Moreover, we have set
$$\CL^{1-2H} f = {1\over \Gamma(2H-1)}\int_0^\infty t^{2H-2} P_tf\dd t.$$
The last integral is clearly integrable at $t=0$ while integrability at $t=\infty$ follows from an 
ergodic assumption on $\CL$, namely Assumption~\ref{ass:ergodic}.

\begin{theorem}\label{thm:main}
  Suppose that $f$ and $g$ satisfy Assumption~\ref{ass:C3} and let $x_t^\epsilon$ denote the solution to
  \begin{equation}\label{main-equation}
    \dd x_t^\epsilon = f(x_t^\epsilon, y^\epsilon_t) \dd B^H_t + g(x_t^\epsilon, y^\epsilon_t) \dd t,
    \qquad x_0^\epsilon=x_0.
  \end{equation}
  Assume that $y_t$ satisfies Assumption~\ref{ass:ergodic}. 
  Then the following statements hold for any $\alpha \in \left(0, \tfrac 12\right)$ and $\alpha' \in (0, H)$.
  \begin{enumerate}
    \item Denoting $V^\epsilon, V$ as in \eqref{e:v-rp} and \eqref{eq:V-covariance}, $(V^\epsilon, B^H)$ converges in law to $(V, B^H)$ in
          $C^\alpha([0,T]; C^3_b) \times C^{\alpha'}([0,T]; \R^m)$.
    \item Furthermore, the fluctuation process from the solution $\bar x$ of the averaged equation \eqref{eq:bar-x}
          \begin{equ}\label{define-z}
            z_t^\epsilon := \epsilon^{\frac{1}{2}-H} \left(x^\epsilon_t - \bar x_t\right)
          \end{equ}
          converges in law in $C^\alpha([0,T]; \mathbb{R}^d)$ to the unique solution of the equation
          \begin{equation}\label{eq:z-lim-eq}
            z_t =  \int_0^t V(\dd s, \bar x_s) + \int_0^t D\bar f(\bar x_s) z_s \dd B^H_s
            + \int_0^t D\bar g(\bar x_s) z_s \dd s,
          \end{equation}
          where the first integral is understood in the nonlinear Young sense of
          Lemma~\ref{lem:NY-int} below and the second in the (standard) Young sense.
  \end{enumerate} \end{theorem}

\begin{remark}
  We remark that the Gaussian field $V$, and in particular its spatial covariance $\Sigma$,
  depends only on $f$ and not on $g$. The vanishing of the fluctuations from the average drift $\bar g$ under 
  the scaling $\epsilon^{\f 12 - H}$ is presented in Appendix~\ref{sec:drift-vanish}.
\end{remark}

We briefly comment on the difference between the present result and the techniques used in
\cite{Hairer-Li:22}. First, even in the case $\bar f=0$, and $g=0$ when the averaged limit is trivial, we note that
while \cite{Hairer-Li:22} obtains a similar Gaussian limit, our main theorem does not follow from \cite{Hairer-Li:22}.

Let us also note that the infinite dimensional rough path approach used in \cite{Hairer-Li:22} can be viewed
as a counterpart to the approach taken here. In particular, this approach was also used in the previous 
version of our article \cite{arxiv}. There, by extending the analysis of \cite{Hairer-Li:22}, together 
with additional rough path arguments, we proved Theorem~\ref{thm:main}, in the case $g$ identically zero.

However, as was pointed out by one of the referees, in our case, it is possible to view the relevant equations
as nonlinear Young differential equations. This approach appears more direct and accessible and we
have adopted it in this version of the paper. Moreover, we have extended the main result to include a non-trivial drift $g$, as was
also pointed out by the referee. In some sense, the approach used in
\cite{Hairer-Li:22} which covers also $H\in (\f 13,\f 12)$ can be viewed as a rough counterpart of the
approach taken here. We note that nonlinear Young integration was introduced in \cite{Catellier-Gubinelli:16}, further developed in
\cite{Hu-Le:17, Galeati:23} with an extension to the rough regime studied in \cite{Nualart-Xia:20}.

In this article, Theorem~\ref{thm:main} is proved by establishing the
joint convergence in law of the pair $(V^\epsilon, B^H)$ to a jointly Gaussian limit with independent components.
Then, viewing $z^\epsilon$ as the solution to a random, time-dependent,
linear Young differential equation (YDE) driven by $B^H$ with an additional forcing term given by the nonlinear Young integral
$\int_0^\cdot V^\epsilon(\dd s, x^\epsilon_s)$, the required convergence follows from
the continuous mapping theorem. The final continuity estimate is provided by the residue Lemma~\ref{lem:residue-yde}.
The first part of Theorem~\ref{thm:main} is proved in Theorem~\ref{thm:VB-conv} while the second part is presented
in Section~\ref{subsec:proof-main-theorem} after establishing the convergence of the nonlinear
Young integral $\int_0^\cdot V^\epsilon(\dd s, x^\epsilon_s)$.

\subsection{Proof strategy}\label{sec:proof-strat}
In this subsection we outline the main steps in the proof of Theorem~\ref{thm:main}.

{\bf A.} We begin by deriving an expression for the rescaled difference process
$$z_t^\epsilon = \epsilon^{\frac{1}{2}-H} \left(x_t^\epsilon -\bar x_t \right)
  = \epsilon^{\frac{1}{2}-H} \int_0^t \left[f(x_s^\epsilon, y_s^\epsilon) - \bar{f}(\bar x_s)\right] \dd B^H_s
  + \epsilon^{\frac{1}{2}-H} \int_0^t \left[g(x_s^\epsilon, y_s^\epsilon) - \bar{g}(\bar x_s)\right] \dd s.$$
To simplify the expressions inside the integrals, we Taylor expand $\bar f$ and $\bar g$ around $x^\epsilon_s$:
\begin{equ}
  \bar{f}( \bar x_s) = \bar{f}(x_s^\epsilon)
    + \int_0^1 D\bar f (\theta x^\epsilon_s + (1 - \theta)\bar x_s)(\bar x_s - x^\epsilon_s) \dd \theta,
\end{equ}
and similarly for $\bar g$. This allows us to rewrite $z_t^\epsilon$ as
\begin{equation}\label{eq:z-epsilon-YDE}
  \begin{split}
  z_t^\epsilon &= \epsilon^{\frac{1}{2}-H} \int_0^t \left[f(x_s^\epsilon, y_s^\epsilon) - \bar{f}(x_s^\epsilon)\right] \dd B^H_s
    + \int_0^t \left(\int_0^1 D\bar f (\theta x^\epsilon_s + (1 - \theta)\bar x_s) \dd \theta\right) z^\epsilon_s \dd B^H_s\\
  &\qquad + \epsilon^{\frac{1}{2}-H} \int_0^t \left[g(x_s^\epsilon, y_s^\epsilon) - \bar{g}(x_s^\epsilon)\right] \dd s
    + \int_0^t \left(\int_0^1 D\bar g (\theta x^\epsilon_s + (1 - \theta)\bar x_s) \dd \theta\right) z^\epsilon_s \dd s.
  \end{split}
\end{equation}
Setting
\begin{equation}
  A^\epsilon_t := \int_0^1 D\bar f (\theta x^\epsilon_t + (1 - \theta)\bar x_t) \dd \theta, \qquad
  \tilde A^\epsilon_t := \int_0^1 D\bar g (\theta x^\epsilon_t + (1 - \theta)\bar x_t) \dd \theta,
\end{equation}
and denoting
\begin{equation}\label{eq:V-epsilon-def}
  w^\epsilon_t := \epsilon^{\f 12 - H} \int_0^t \left[f(x^\epsilon_r, y^\epsilon_r) - \bar f(x^\epsilon_r)\right] \dd B^H_r,
  \qquad
  \tilde w^\epsilon_t := \epsilon^{\f 12 - H} \int_0^t \left[g(x^\epsilon_r, y^\epsilon_r) - \bar g(x^\epsilon_r)\right] \dd r,
\end{equation}
we get
\begin{equ} \label{eq:z-eps-relation}
  z_t^\epsilon = w^\epsilon_t + \tilde w^\epsilon_t + \int_0^t A^\epsilon_s z^\epsilon_s \dd B^H_s
  + \int_0^t \tilde A^\epsilon_s z^\epsilon_s \dd s.
\end{equ}
The crux now lies in understanding the forcing terms $w^\epsilon$ and $\tilde w^\epsilon$, and comparing
the above expression to its limiting counterpart~\eqref{eq:z-lim-eq}.

{\bf B.}
As \cite{Hairer-Li:20} demonstrates the convergence in probability of $x^\epsilon$ to
$\bar x$ in the H\"older topology, the convergence of $A^\epsilon, \tilde{A}^\epsilon$ is straightforward and we
focus on the convergence of the forcing terms $w^\epsilon, \tilde{w}^\epsilon$ as $\epsilon \to 0$.
Once the function-valued process $V^\epsilon$ is defined, see \eqref{e:v-rp}, we may use the
nonlinear Young integral construction
\begin{equation}\label{eq:w-eps-NY}
 w^\epsilon_t = \int_0^t V^\epsilon(\dd s, x^\epsilon_s) := \lim_{|\pi| \to 0} \sum_{[u, v] \in \pi} V^\epsilon_{u, v}(x^\epsilon_u)
\end{equation}
for which the limit exists as the components are sufficiently regular.

Now, as the map $(V, x) \mapsto \int_0^\cdot V(\dd s, x_s)$ is continuous in the corresponding
H\"older topologies, the convergence of $w^\epsilon$ reduces to the \textit{joint} convergence in law of the pair
$(V^\epsilon, x^\epsilon)$ and ultimately to that of the pair $(V^\epsilon, B^H)$.
Establishing the joint convergence of $(V^\epsilon, B^H)$, together with the independence of the limiting
components, is the content of Section~\ref{sec:V-epsilon-cvg}.

The same point of view applies to the drift fluctuation term $\tilde w^\epsilon$. Setting
\begin{equation}\label{eq:G-eps-def-intro}
  G^\epsilon_t(x) := \epsilon^{\f 12 - H} \int_0^t \left[g(x, y^\epsilon_r) - \bar g(x)\right] \dd r,
\end{equation}
one again defines $\tilde w^\epsilon = \int_0^\cdot G^\epsilon(\dd s, x^\epsilon_s)$ as a nonlinear Young
integral. In contrast with $V^\epsilon$, whose scaling is adapted to the fluctuations created by the
fractional noise, the fluctuations of the drift are governed by the classical central limit
theorem scaling $\epsilon^{- \f 12}$. In Lemma~\ref{lem:G-eps-vanish} we show that $\|G^\epsilon\|_{C^\beta([0,T]; C^3_b)} \to 0$ in $L^p$ with rate $\epsilon^{1 - H}$, so that $\tilde w^\epsilon \to 0$ in probability.

{\bf C.}
By the previous steps, the processes $(w^\epsilon, \tilde w^\epsilon, A^\epsilon, \tilde A^\epsilon, B^H)$
converge jointly in law to
$$\left(\int_0^\cdot V(\dd s, \bar x_s),\; 0, \; D\bar f(\bar x), D\bar g(\bar x),\; B^H\right).$$
We now view
\begin{equ}
  z_t^\epsilon = w^\epsilon_t + \tilde w^\epsilon_t + \int_0^t A^\epsilon_s z^\epsilon_s \dd B^H_s
  + \int_0^t \tilde A^\epsilon_s z^\epsilon_s \dd s.
\end{equ}
 as a linear YDE for $z^\epsilon$, with random coefficient
$(A^\epsilon, \tilde A^\epsilon)$, driver $X_t = (B^H_t, t)$, together with the forcing
$w^\epsilon + \tilde w^\epsilon$. It therefore remains to show that the solution map
$(w, A, X) \mapsto z$ associated to the linear YDE
\begin{equation*}
  z_t = w_t + \int_0^t A_s z_s \dd X_s
\end{equation*}
is continuous in the appropriate H\"older topologies, which will be provided by the
deterministic residue Lemma~\ref{lem:residue-yde}. A final application of the continuous
mapping theorem identifies the limit of $z^\epsilon$ as the solution of the limiting equation
\eqref{eq:z-lim-eq}.
This is the content of Section~\ref{sec:proof-main}.
\begin{remark}
  We make a brief remark regarding the choice of the decomposition~\eqref{eq:z-epsilon-YDE}. It might
  be tempting to consider the alternative decomposition in which the Taylor expansion is performed
  around $\bar x$. Namely, taking $V^\epsilon$ as before, one can instead write (assuming for simplicity $g = 0$)
  $$z_t^\epsilon = \int_0^t V^\epsilon(\dd s, \bar x_s) + \int_0^t \hat A^\epsilon_s z_s^\epsilon \dd B^H_s$$
  where $\hat A^\epsilon_t$ is now defined as
  $$\hat A^\epsilon_t := \int_0^1 D f(\theta x^\epsilon_t + (1 - \theta)\bar x_t, y^\epsilon_t) \dd \theta.$$
  This alternative decomposition has the advantage that the dependence on $\epsilon$ in the forcing
  term now solely rests on $V^\epsilon$, so that its convergence becomes much simpler.
  Nonetheless, this advantage comes at the cost that the integral
  $\int_0^t \hat A^\epsilon_s z_s^\epsilon \dd B^H_s$ can no longer be interpreted path-wise.
  This is in contrast to the previous decomposition in which the corresponding integral can be simply interpreted
  in the Young sense. Instead, this integral can only be made sense of via the stochastic sewing Lemma
  (cf. \cite[Section 3.3]{Hairer-Li:20} and also \cite{Le:20}) from which one can only obtain probabilistic estimates. 
  This presents a difficulty since the residue Lemma we rely on is inherently
  path-wise and upgrading it to an estimate in moments would require bounds of the form
  $\sup_\epsilon\mathbb{E}\exp \|x^\epsilon\|_{\frac{1}{2}-}^\beta < \infty$ for some $\beta > 2$.
  This would consequently demand unrealistically strong assumptions on~$f$.
\end{remark}

\subsection{Notations}\label{sec:notations}
We provide here the notations used throughout this article.
\begin{itemize}
  \item $(\Omega, \mathcal{F}, \mathbb{P})$ is a probability space, and $\|\cdot\|_p$
        denotes the norm in $L^p(\Omega;\mathbb{P})$.
  \item Generic separable Banach spaces are denoted by calligraphic letters $\CU, \CV, \CW$.
  \item For a one-parameter process $x : [0,T] \to \CV$ and $s\leq t$ we denote $x_{s,t} := x_t-x_s.$
  \item We write $\Delta_T := \{(s,t) \in [0,T]^2 : s \leq t\}$.
  \item For a two-parameter process $A$ on $\Delta_T$, and $0<s<u<t\leq T$ we define 
    $$\delta A_{s,u,t} := A_{s,t} -A_{s,u} -A_{u,t}. $$ 
    We also set $$\|A\|_\alpha := \sup_{(s,t) \in \Delta_T}\frac{\|A_{s,t}\|}{|t-s|^\alpha}.$$
  \item For a path $x \in C^\alpha([0,T]; \CV)$ and a subinterval $[s,t] \subseteq [0,T]$ we write
    $\|x\|_{\alpha; [s,t]}$ and $\|x\|_{\infty; [s,t]}$ for the H\"older seminorm and supremum norm
    of the restriction of $x$ to $[s,t]$.
  \item For a function-valued path $V \in C^\beta([0,T]; C^k_b)$, we use the shorthand
    $\|V\|_{\beta; k} := \|V\|_{C^\beta([0,T]; C^k_b)}$, where the H\"older norm is in time.
  \item We use the notation $\|f\|_{L^pC^\alpha} := \|\|f\|_{C^\alpha}\|_{L^p}$ where the H\"older
    norm is in time, and $L^p$ in probability, with the underlying spaces being clear from the context.
  \item $\L(\CU, \CV)$ denotes the space of bounded linear operators between two linear spaces $\CU$ and $\CV$.
  \item $C_b^k(\CU)$ denotes the space of real valued bounded $C^k$ functions on $\CU$ with continuous and bounded Fréchet derivatives up to order $k$.
\end{itemize}

\subsection{Preliminaries on nonlinear Young integration}\label{sec:nonlinear-young}
Nonlinear Young integration is based on the sewing Lemma. We recall here the properties used below,
following \cite{Catellier-Gubinelli:16, Hu-Le:17, Galeati:23}.

\begin{lemma}[Nonlinear Young integral]\label{lem:NY-int}
  Let $\alpha,\beta \in (0,1)$ with $\alpha + \beta > 1$,
   $V \in C^\beta([0,T]; C^1_b(\R^d, \R^e))$ and
  $x \in C^\alpha([0,T]; \R^d)$. Then the limit
  \begin{equation}\label{eq:NY-def}
    \int_0^t V(\dd s, x_s) := \lim_{|\pi| \to 0} \sum_{[u, v] \in \pi} V_{u, v}(x_u)
  \end{equation}
  exists along partitions $\pi$ of $[0,t]$. The integral $ \int_0^t V(\dd s, x_s)$ defines a path in $C^\beta([0,T]; \R^e)$ 
  and satisfies the estimate
  \begin{equation}\label{eq:NY-local-est}
    \left| \int_s^t V(\dd r, x_r) - V_{s,t}(x_s) \right|
    \le C \|V\|_{\beta;1} \|x\|_\alpha |t-s|^{\alpha + \beta},
  \end{equation}
  where $C$ depends only on $\alpha + \beta$. 
  Moreover, if $V_t(x) = \varphi(x) b_t$ for some $\varphi \in C^1_b(\R^d, \L(\R^m, \R^e))$ and
  $b \in C^\beta([0,T]; \R^m)$, then the nonlinear Young integral agrees with the Young integral, i.e.
  $$\int_0^t V(\dd s, x_s) = \int_0^t \varphi(x_s) \dd b_s.$$
\end{lemma}

\begin{lemma}[Stability of nonlinear Young integrals]\label{lem:NY-stability}
  Let $\alpha + \beta > 1$, let $x, \tilde x \in C^\alpha([0,T]; \R^d)$ and let $V, \tilde V$ be
  paths in $C^\beta([0,T]; C^2_b(\R^d, \R^e))$. Then, writing $w := \int_0^\cdot V(\dd s, x_s)$ and
  $\tilde w := \int_0^\cdot \tilde V(\dd s, \tilde x_s)$, we have
  \begin{equation}\label{eq:NY-stability}
    \|w - \tilde w\|_\beta
    \le C\left(1 + \|x\|_\alpha + \|\tilde x\|_\alpha \right)
      \left( \|V - \tilde V\|_{\beta; 1}
      + \|\tilde V\|_{\beta;2} \left(\|x - \tilde x\|_\infty + \|x - \tilde x\|_\alpha\right)\right),
  \end{equation}
  where $C$ depends only on $\alpha + \beta$ and $T$. In particular, the map
  $(V, x) \mapsto \int_0^\cdot V(\dd s, x_s)$ is continuous from
  $C^\beta([0,T]; C^2_b) \times C^\alpha([0,T]; \R^d)$ to $C^\beta([0,T]; \R^e)$.
\end{lemma}

We will also make use of the continuity of the It\^o map for classical Young differential
equations, which we recall from the literature (see also \cite{Li-Ying:25} for extensions to
coefficients with unbounded derivatives).

\begin{lemma}[cf. Proposition 4 of \cite{Lejay:10}]\label{lem:young-ito-map}
  Let $\alpha \in (\f 12, 1)$ and $\varphi \in C^2_b(\R^d, \L(\R^m, \R^d))$. For every
  $b \in C^\alpha([0,T]; \R^m)$ and $x_0 \in \R^d$, the YDE
  $$x_t = x_0 + \int_0^t \varphi(x_s) \dd b_s$$
  admits a unique solution $x \in C^\alpha([0,T]; \R^d)$, and the It\^o map
  $F : C^\alpha([0,T]; \R^m) \to C^\alpha([0,T]; \R^d)$, $b \mapsto x$, is
  continuous.
\end{lemma}
\begin{remark}\label{rem:little-holder}
  The Hölder spaces appearing above are not separable. However, by the spatial decay imposed in Assumption~\ref{ass:C3}, 
  all the function-valued processes considered below take values almost surely in a fixed separable Banach
  subspace $\mathcal B\subset C_b^3(\mathbb R^d;\mathbb R^d)$. More precisely, one may take
  $$\mathcal B = C_0^3(\mathbb R^d;\mathbb R^d) := 
    \left\{\phi \in C^3_b : \lim_{|x| \to \infty} D^k \phi(x) = 0,\ 0 \le k \le 3\right\}.$$ 
  Moreover, as time-indexed paths, they 
  belong almost surely to the little Hölder space $C^{0,\beta}([0,T];\mathcal B)$ defined as the closure 
  of smooth paths under the $C^\beta$ norm. All weak-convergence statements are first understood in this separable
  space. Since the canonical embedding
  $$C^{0,\beta}([0,T];\mathcal B) \hookrightarrow C^\beta([0,T];C_b^3)$$
  is continuous, they also imply the stated convergence in the larger Hölder space
  (cf. also \cite{Hairer-Li:22} in which the same considerations were made).
\end{remark}

\subsubsection{Ergodic and regularity conditions}\label{sec:assumptions}
We state the spectral gap assumption imposed on the Markov process $y$ on $\CY$.
Let $E_0\subset E_1\subset E_2\subset \dots \subset  L^2(\CY,\mu)$
be a sequence of Banach spaces containing the constants such that
pointwise multiplication $$E_0 \times E_n \to E_{n+1}$$ is  continuous, for every $n\ge 0$. Note 
that $E_0 \subset \cap_{p\ge 1} L^p$.

\begin{assumption}\label{ass:ergodic}
  Assume $y$ is continuous and that for each $n\ge 1$ the semigroup $P_t$ associated to $y$ extends to a strongly 
  continuous semigroup on $E_n$ and there exist constants $C$ and $c> 0$ (possibly depending on $n$) 
  such that, for every $f \in E_n$ with $\int_\CY f \dd \mu=0$, one has
  \begin{equ}[e:spectralGap]
    \|P_t f\|_{E_n} \le C e^{-ct} \|f\|_{E_n}\;.
  \end{equ}
\end{assumption}

\begin{example}\label{ex:ou}
  For the Ornstein-Uhlenbeck process $dy_t=-Ay_t+dW_t$ where $A$ is a symmetric positive 
  matrix, we take $E_0 = C_b(\mathbb R^d)$ and, for $n \ge 1$, $E_n = C_V^0$, where
  $$C_V^0 := \left\{f\in C(\mathbb R^d) :
    \|f\|_V := \sup_{y\in\mathbb R^d}\frac{|f(y)|}{V(y)}<\infty \text{ and }
    \lim_{|y|\to\infty}\frac{|f(y)|}{V(y)} = 0\right\},$$
  and $V(x)=1+ |x|^2$.  Then $V$ is a Lyapunov function, and the spectral gap 
  property~\eqref{e:spectralGap} follows from the Foster–Lyapunov criterion. Moreover, the Ornstein-Uhlenbeck 
  semigroup leaves $C_V^0$ invariant and is strongly continuous on this space. This argument extends to 
  more general strong Feller SDEs on a Banach space that admit a Lyapunov function and are irreducible
  \cite{HarrisOrig, Harris, HMS}.
\end{example}

Note that although the Ornstein-Uhlenbeck process on $\R^d$ converges in the total variation norm, 
it does not satisfy the $L^\infty$ spectral gap condition. The rate of convergence of its densities, 
$|p_t(x,y)-p(y)|$,  is not uniform in $x$; hence exponential ergodicity cannot be extended to the 
$L^\infty$ norm. Consider a smooth function that vanishes in $B_1$ (the ball of radius $1$) and 
takes the value $1$ on $B_2^c$. Because the Ornstein–Uhlenbeck process is a $C_0$-diffusion,
$\lim_{x\to \infty} P_tf(x)= 1$ for any $t>0$, while $\int_{\R^d} f \dd \mu<1$, and thus  $P_t f$ does 
not converge to $\int_{\R^d} f \dd \mu$ in $L^\infty$.

\begin{example}\label{ex:compact-manifold}
  Let $\cY$ be a connected compact Riemannian manifold and let $F_0, F_1, \dots, F_{\hat m}$ be
  smooth vector fields on $\cY$ such that $F_1(y), \dots, F_{\hat m}(y)$ span the tangent space at
  every $y \in \cY$. Consider the stationary solution of the SDE
  $$\dd y_t = F_0(y_t) \dd t
    + \sum_{k=1}^{\hat m} F_k(y_t)\circ \dd W_t^k\;,$$
  where $W_t=(W_t^1, \dots, W_t^{\hat m})$  is an $\hat m$-dimensional standard Wiener process
  independent of $B^H$. The associated generator
  $\CL = F_0 + {1\over 2} \sum_{k=1}^{\hat m} (F_k)^2$ is uniformly elliptic, so that $y$ admits a
  unique invariant measure $\mu$, with a smooth, everywhere positive density. By the compactness of
  $\cY$, the transition densities at time one are bounded from below, uniformly over the starting
  point, by a multiple of $\mu$. Doeblin's criterion (cf.\ \cite{Harris}) therefore yields
  constants $C, c > 0$ such that, for every bounded measurable $f$ with $\int_\cY f \dd \mu = 0$,
  $$\|P_t f\|_{L^\infty} \le C e^{-ct} \|f\|_{L^\infty}, \qquad t \ge 0.$$
  Since moreover $(P_t)$ is Feller, and hence strongly continuous on $C^0(\cY)$, Assumption
  \ref{ass:ergodic} holds with the constant sequence $E_n = C^0(\cY)$, which is a Banach algebra
  containing the constants and contained in $L^1(\cY, \mu)$.

  We remark that on a \textit{non-compact} state space, uniform ellipticity alone does not
  suffice for Assumption~\ref{ass:ergodic}: for instance, $\CL = \Delta$ on $\R^d$ has no
  spectral gap in any of the spaces above. In that situation a confining drift is required, as in
  Example~\ref{ex:ou}, where the spectral gap is produced by a Lyapunov function in combination
  with the Harris theorem.
\end{example}

We also assume conditions on the functions $f$ and $g$.
\begin{assumption}\label{ass:C3}
  The maps $x \mapsto f(x,\cdot)$ and $x \mapsto g(x,\cdot)$  are of class $C^4$ with values in $E_0\subset L^2(\cY,\mu)$ and 
  there exists an exponent $\kappa > 0$ such that, for every multi-index $\ell$ of length at most $4$,
  $$\|D_x^\ell f(x,\cdot)\|_{E_0} \lesssim (1+|x|)^{-\kappa}, \quad \|D_x^\ell g(x,\cdot)\|_{E_0} \lesssim (1+|x|)^{-\kappa}.$$
\end{assumption}

\begin{remark}
  We expect the regularity and decay assumptions to not be sharp, and of purely technical nature. We keep them in the same form as in \cite{Hairer-Li:22}, in the $H>\frac{1}{2}$ regime covered there.
\end{remark}

\section{Joint convergence of \texorpdfstring{$(V^\epsilon, B^H)$}{(V, B)}}\label{sec:V-epsilon-cvg}

In this section we justify Step \textbf{B} of the proof outline in Section~\ref{sec:proof-strat} by
establishing the joint convergence in law of the pair $(V^\epsilon, B^H)$, together with the
independence of the components of the limit. 
Let us fix $\alpha = H - \f 12$ throughout this section, and recall
\begin{equ}
  \begin{aligned}
    V_{s,t}^\epsilon := x \mapsto V^{\epsilon}_{s,t}(x)= \epsilon^{-\alpha}\int_s^t \left[f(x, y^\epsilon_r)
      - \bar f(x)\right] \dd B^H_r, \quad x\in \mathbb{R}^d.
  \end{aligned}\end{equ}
in which the integral is understood as a Wiener integral.

The main result of this section is the following and constitutes the first statement of Theorem~\ref{thm:main}.

\begin{theorem}\label{thm:VB-conv}
  Suppose Assumptions~\ref{ass:ergodic} and \ref{ass:C3} hold, and let $\beta \in (0, \f 12)$,
  $\alpha' \in (0, H)$. Then, as $\epsilon \to 0$,
  $$(V^\epsilon, B^H) \stackrel{(d)}{\longrightarrow} (V, B^H)
    \quad \text{in } C^\beta([0,T]; C^3_b) \times C^{\alpha'}([0,T]; \R^m),$$
  where $V$ is the Gaussian field defined in Theorem~\ref{thm:main} and is independent of
  $B^H$.
\end{theorem}

This will come as a consequence of Lemma~\ref{lem:cumulant-conv} from Subsection~\ref{sec:cumulant}, 
together with the following results regarding $V^\epsilon$ which follow directly as a consequence of
\cite[Lemma 4.2, Lemma A.1 and Proposition 5.1]{Hairer-Li:22}.
\begin{proposition}[\cite{Hairer-Li:22}]\label{prop:V-eps-moments}
  Let $f$ satisfy Assumption~\ref{ass:C3}, let $y_t$ satisfy Assumption~\ref{ass:ergodic} and let
  $\beta \in (0, \f 12)$. Then, for every $p \ge 1$,
  \begin{equation}\label{eq:V-eps-unif-bd}
    \sup_{\epsilon \in (0, 1]} \|V^\epsilon\|_{L^p C^\beta} < \infty.
  \end{equation}
  Moreover, $V^\epsilon$ converges in law to the centered Gaussian field $V$ which has covariance
  \begin{equation*}
    \mathbb{E}[V_{s, t}(x) \otimes V_{u, v}(x')]
    = |[s, t] \cap [u, v]|(\Sigma(x, x') + \Sigma(x', x)^{\top})
  \end{equation*}
  where $\Sigma$ is explicitly defined by \eqref{e:defSigma}. 
\end{proposition}

\begin{proof}[of Theorem~\ref{thm:VB-conv}]
   Since $V^\epsilon$ is tight by Proposition~\ref{prop:V-eps-moments} while $B^H$ is independent of $\epsilon$, 
   the pair $(V^\epsilon, B^H)$ is tight in the product topology. Keeping in mind Remark~\ref{rem:little-holder}, 
   it therefore suffices to identify the
  limiting finite dimensional distributions, namely the joint limit in law of the random vectors
  \begin{equation}\label{eq:fdd-vector}
    \left(V^{\epsilon}_{s_1, t_1}(x_1)_{i_1}, \dots, V^{\epsilon}_{s_n, t_n}(x_n)_{i_n},\;
    B^{H}_{u_1, v_1}, \dots, B^{H}_{u_{n'}, v_{n'}}\right)
  \end{equation}
  for arbitrary $n, n' \in \mathbb{N}$, intervals $[s_j, t_j], [u_j, v_j] \subseteq [0, T]$, points
  $x_j \in \R^d$ and coordinate indices $i_j \in \{1, \dots, d\}$. By the Cram\'er-Wold theorem, it
  suffices in turn to show that an arbitrary linear combination of the entries of
  \eqref{eq:fdd-vector} converges in law to a Gaussian limit and that the covariances converge. 
  Convergence to a Gaussian limit
  is the content of Lemma~\ref{lem:cumulant-conv} which we prove in the following section, while 
  independence follows from Lemma~\ref{lem:cross-cov}, which shows that the cross covariances 
  vanish identically giving the independence of $V$ and $B^H$ thanks to Gaussianity.
\end{proof}

\subsection{Cumulant computations for one-dimensional projections}\label{sec:cumulant}

In this subsection we prove Lemma~\ref{lem:cumulant-conv} as well as Lemma~\ref{lem:cross-cov} and therefore conclude the 
proof of Theorem~\ref{thm:VB-conv}. Namely, we show that arbitrary linear combinations of the entries of the random
vectors \eqref{eq:fdd-vector}, that is, of point evaluations of $V^\epsilon$ over arbitrary time
intervals together with arbitrary increments of $B^H$, converge to a Gaussian limit in the sense
of cumulants. To be precise, we prove that their joint cumulants of order higher than two vanish
in the limit, while those of order two converge. This allows us to conclude that the finite
dimensional distributions of $(V^\epsilon, B^H)$ converge weakly to a Gaussian limit. Indeed,
convergence of cumulants implies convergence of moments, as can be seen from the formula
\eqref{cul-relation2}. By \cite[Theorem 30.2]{Billingsley:95}, this implies weak convergence when the limit is determined by its moments. This holds for real Gaussians (see Example 30.1 in \cite{Billingsley:95}), and extends directly to the multivariate
setting.

 Lemma~\ref{lem:cross-cov} below shows that cross covariances vanish identically.

\begin{lemma}\label{lem:cross-cov}
  For every $\epsilon > 0$, $(s, t), (u, v) \in \Delta_T$, $x \in \R^d$ and $i \in \{1, \dots,
  d\}$, we have
  $$\E\left[V^\epsilon_{s,t}(x)_i B^{H}_{u, v}\right] = 0.$$
\end{lemma}
\begin{proof}
  Writing $h(\cdot) := f_{i}(x, \cdot) - \bar{f}_{i}(x)$, which satisfies $\int_\CY h \dd \mu = 0$,
  we compute
  \begin{align*}
    \E\left[V^\epsilon_{s,t}(x)_i B^{H}_{u, v}\right]
     & = \epsilon^{\f 12 - H}\mathbb{E}\left[B^{H}_{u, v}
    \int_{s}^{t} h(y^\epsilon_r) \dd B^H_r\right]                            \\
     & = \epsilon^{\f 12 - H}\mathbb{E}\left[B^{H}_{u, v}
    \int_{s}^{t} \mathbb{E}[h(y^\epsilon_r) \mid \sigma(B^H)] \dd B^H_r\right] \\
     & = \epsilon^{\f 12 - H}\mathbb{E}\left[B^{H}_{u, v}
    \int_{s}^{t} \mathbb{E}[h(y^\epsilon_r)] \dd B^H_r\right] = 0,
  \end{align*}
  where we used the independence of $B^H$ and $y^\epsilon$ in the penultimate step, and the
  stationarity of $y$, in the last step.
\end{proof}

We are left to show the joint Gaussianity, which will constitute the rest of this subsection.

To lighten the notation, we present all the subsequent computations for a scalar fractional
Brownian motion, $m = 1$. The general case differs only by keeping track of the
components of $B^H$, cf. \cite{Hairer-Li:22}.

Throughout this section, we adopt the notations of \cite{Hairer-Li:22}.
For a finite index set $B$, let $\mathcal{P}(B)$ denote the set of partitions of $B$ and we
equip it with the partial order given by refinement of partitions. Namely, for
$\Delta, \tilde \Delta \in \mathcal{P}(B)$, we write $\Delta \le \tilde \Delta$ if for any set
$A \in \Delta$ there exists $\tilde A \in \tilde \Delta$ with $A \subset \tilde A$. This induces a
lattice structure on $\mathcal{P}(B)$ with the join operator $\vee$ defined such that $\Delta \vee \tilde \Delta$
is the finest partition greater than both $\Delta$ and $\tilde \Delta$. Then, we denote
$\mathcal{G}^S(B)$ by the set of pairs $(\Delta, p)$, where
$\{\Delta, p\} \subseteq \mathcal{P}(B)$ for which $p$ consists only of pairs and $\Delta \vee p = \{B\}$.
We remark that the latter condition is known as \textit{connectedness} of the partitions
with an example and a counter-example given by Figure~\ref{fig:con-part} and~\ref{fig:non-con-part}.

\begin{figure}[h!]
  \centering
  \begin{minipage}[t]{0.45\textwidth}
    \centering
    \begin{tikzpicture}[baseline=7.5ex,scale=0.9]
      \foreach \i in {1, ..., 6}
        {
          \pgfmathtruncatemacro{\x}{(\i-1)/2};
          \pgfmathtruncatemacro{\y}{\i - 2*\x};
          \coordinate (\i) at (\x,\y);
        }
      \draw[partition] (1) -- (2) -- (4);
      \draw[partition] (3) -- (5) -- (6);
      \foreach \i in {1, 2, 3}
        {
          \pgfmathtruncatemacro{\pairstart}{2*\i-1};
          \pgfmathtruncatemacro{\pairend}{2*\i};
          \draw[pairing] (\pairstart) -- (\pairend);
        }
      \foreach \i in {1, ..., 6}
        {
          \coordinate[label={[font=\tiny, left]:$\i$}] (a) at (\i);
        }
    \end{tikzpicture}
    \caption{A connected partition.}
    \label{fig:con-part}
  \end{minipage}
  \begin{minipage}[t]{0.45\textwidth}
    \centering
    \begin{tikzpicture}[baseline=7.5ex,scale=0.9]
      \foreach \i in {1, ..., 6}
        {
          \pgfmathtruncatemacro{\x}{(\i-1)/2};
          \pgfmathtruncatemacro{\y}{\i - 2*\x};
          \coordinate (\i) at (\x,\y);
        }
      \draw[partition] (1) -- (2) -- (4) -- (3) -- (1);
      \draw[partition] (5) -- (6);
      \coordinate (mid) at (1/2,3/2);
      \draw[partition] (mid) -- (mid);
      \foreach \i in {1, 2, 3}
        {
          \pgfmathtruncatemacro{\pairstart}{2*\i-1};
          \pgfmathtruncatemacro{\pairend}{2*\i};
          \draw[pairing] (\pairstart) -- (\pairend);
        }
      \foreach \i in {1, ..., 6}
        {
          \coordinate[label={[font=\tiny, left]:$\i$}] (a) at (\i);
        }
    \end{tikzpicture}
    \caption{A disconnected partition.}
    \label{fig:non-con-part}
  \end{minipage}
\end{figure}

Finally, for an index set $B$ we write $X^B = \prod_{a \in B} X_a$, $X_B = (X_a)_{a \in B}$ and 
$\mathbb{E}_c X_B$ for the joint cumulant of the collection $X_B$.

Recall that the cumulants can be expressed in terms of moments by the recursive relation
\begin{equ}\label{cul-relation1}
  \mathbb{E}_c[X_B]
  = \sum_{\Delta \in \mathcal{P}(B)} (|\Delta| - 1)!  (-1)^{|\Delta| - 1}
  \prod_{A \in \Delta} \mathbb{E}[X_A].
\end{equ}
Conversely, we can express the moments in terms of cumulants as follows
\begin{equ}\label{cul-relation2}
  \mathbb{E}[X^B]
  = \sum_{\Delta \in \mathcal{P}(B)} \prod_{A \in \Delta} \mathbb{E}_c[X_A],
\end{equ}
where the sum runs over all partitions $\Delta$ of $B$, and for each block $A \in \Delta$,
$\mathbb{E}_c[X_A]$ denotes the cumulant of the sub-collection $X_A$.
For small index sets $B$, this gives
$$
  \begin{aligned}
     & \mathbb{E}[X_1] = \mathbb{E}_c[X_1], \qquad
    \mathbb{E}[X_1 X_2] = \mathbb{E}_c[X_1, X_2] + \mathbb{E}_c[X_1]\mathbb{E}_c[X_2], \\
     & \mathbb{E}[X_1 X_2 X_3]
    = \mathbb{E}_c[X_1, X_2, X_3]
    + \mathbb{E}_c[X_1, X_2]\mathbb{E}_c[X_3]
    + \mathbb{E}_c[X_1, X_3]\mathbb{E}_c[X_2]                                            \\
     & \qquad \qquad \qquad  + \mathbb{E}_c[X_2, X_3]\mathbb{E}_c[X_1]
    + \mathbb{E}_c[X_1]\mathbb{E}_c[X_2]\mathbb{E}_c[X_3].
  \end{aligned}
$$
Taking $B$ to a finite index set, we compute the joint cumulants of
$W^\epsilon_B = (W_a^\epsilon : a \in B)$, where each index $a \in B$ carries an interval
$[s_a, t_a] \subseteq [0, T]$, coefficients $\lambda_a, \eta_a \in \R$, a point $x_a \in \R^d$ and a
coordinate $i_a \in \{1, \dots, d\}$, and
\begin{equ}[eq:W_a-def]
  W_a^\epsilon :=
  \lambda_a B^H_{s_a,t_a} + \eta_a \epsilon^{-\alpha} \int_{s_a}^{t_a} \left(f_{i_a}(x_a, y^\epsilon_r) - \bar{f}_{i_a}(x_a)\right) \dd B^H_r,
\end{equ}
where $\alpha = H - \f 12$. Every linear combination
of the entries of \eqref{eq:fdd-vector} is of the form $\sum_{a \in B} W^\epsilon_a$ and, by the
multilinearity of joint cumulants, the cumulants of such sums decompose into sums of the joint
cumulants of $W^\epsilon_B$.

We remark that, since conditionally on $y$ the vector $W^\epsilon_B$ is a centered Gaussian linear
functional of $B^H$, we have $W^\epsilon_B = -W^\epsilon_B$ in law, hence the odd cumulants vanish
and it suffices to compute the even cumulants.

For a pair of partitions $(\Delta, p) \in \mathcal{G}^S(B)$, we denote $S_\Delta = \{A \in \Delta : |A| = 1\}$,
where we drop the subscript $\Delta$ whenever the partition is clear from the context.
For simplicity, we also introduce the shorthands
$$X_a = \eta_a\left(f_{i_a}(x_a, y_{r_a}) - \bar{f}_{i_a}(x_a)\right), \qquad 
  \bar{X}_a = X_a + \lambda_a \epsilon^{\alpha}.$$
Note that $X_a$ and $\bar X_a$ depend on the microscopic time variable $r_a$ rather than $\frac{r_a}{\epsilon}$.
Moreover, we write 
$\dd r^B = \prod_{a \in B} \dd r_a$ and denote the (rescaled) product domain
$$R^B_\epsilon := \prod_{a \in B} \left[\frac{s_a}{\epsilon}, \frac{t_a}{\epsilon}\right].$$
Let us explain how we break down the joint cumulants into integration with respect to fractional kernels. 
First, we compute the moments of $W^\epsilon_B$.
Applying Wick's product formula
conditionally on $y$ (after smoothing $B^H$ with $B^\delta := B^H * \phi_\delta$, where
$\phi_\delta(t) = \delta^{-1}\phi(t/\delta)$ for a Schwartz function $\phi$ integrating to one,
and passing to the limit $\delta \to 0$, cf.\ the proofs of Proposition 4.16 and Lemma 5.5 in
\cite{Hairer-Li:22} for these calculations), followed by the time rescaling
$r_a \mapsto r_a / \epsilon$, we obtain for any finite index set $B$
\begin{equs}\label{cul-relation3}
  \mathbb{E} [(W^\epsilon)^B]
    = C_H^{\frac{|B|}{2}} \epsilon^{\frac{|B|}{2}} \sum_{p}\int_{R^B_\epsilon}
    \E[\bar X^B]\prod_{\{a, b\}\in p}|r_a-r_b|^{2H-2} \dd r^B,
\end{equs}
where the sum runs over all pairings $p$ of $B$ (so that the right-hand side vanishes for $|B|$ odd).
Then, as $\mathbb{E}[\bar X^B]= \sum_{\Delta \in \mathcal{P}(B)} \prod_{A \in \Delta} \mathbb{E}_c[\bar X_A]$, 
one has
\begin{equs}
  \mathbb{E} [(W^\epsilon)^B]
    = C_H^{\frac{|B|}{2}} \epsilon^{\frac{|B|}{2}} \sum_{\Delta, p}\int_{R^B_\epsilon}
     \prod_{A \in \Delta}  \E_c \left[\bar X_A \right]\prod_{\{a, b\}\in p}|r_a-r_b|^{2H-2} \dd r^B.
\end{equs}
From this, by Lemma~\ref{lem:cumulant-decomp} below, we have the decomposition (similar to 
\cite[Lemma 5.4]{Hairer-Li:22}) for the joint cumulants of $W^\epsilon$
\begin{equation}\label{eq:cumulant-decomp}
  \begin{aligned}
    \mathbb{E}_c [W_B^\epsilon] & =  \sum_{(\Delta, p) \in \mathcal{G}^S(B)} C_H^{\frac{|B|}{2}}
    \epsilon^{\frac{|B|}{2}}\int_{R^B_\epsilon}\prod_{A\in \Delta}
    \left(\mathbb{E}_c \bar{X}_A \right) \prod_{\{a, b\}\in p}|r_a-r_b|^{2H-2} \dd r^B             \\
                                & =: \sum_{ (\Delta, p)\in \mathcal{G}^S(B)}\CJ^\epsilon(\Delta, p),
  \end{aligned}\end{equation}
where $C_H = H(2H-1)$ and, inside the integral, $X_a$ and $\bar X_a$ are evaluated at the
(rescaled) times $r_a$.
We remark that, in \cite{Hairer-Li:22} the right-hand side of \eqref{eq:cumulant-decomp} is
instead summed over partitions without singletons, i.e.
$|S_\Delta| = 0$. This is because in their case,  $\CJ^\epsilon(\Delta, p)= 0$ whenever $|S_\Delta| > 0$
and thus is omitted within the sum. 

\begin{lemma}\label{lem:cumulant-decomp}
  Let $B$ be a finite index set and the family 
  $$\{\CJ (\Delta,p): \Delta\in \mathcal P(A), p \hbox{ a pairing in } 
    \mathcal P(A) \text{ for some } A\subset B\}$$ 
  such that
  $$\E[W^A]=\sum_{\Delta\in \mathcal P(A)} \sum_{p \hbox{ a pairing in } P(A)} \CJ(\Delta, p), \quad \forall A\subset B.$$
  Suppose that furthermore $\CJ$ satisfies the factorization property
  $$\CJ(\Delta, p)=\prod _{\hat A\in \Delta\vee p} \CJ(\Delta_{\hat A}, p_{\hat A}), $$
  where $\Delta_{\hat A}, p_{\hat A}$ denote the restrictions of the partitions $\Delta, p$ to the 
  subset $\hat A\subset B$. Then, for any $A\subset B$, $$\E_c[W_A]=\sum_{(\Delta, p)\in \mathcal 
  G^S(A)} \CJ(\Delta, p).$$
\end{lemma}

\begin{proof}
  Given $\Delta,p$ let $\Pi$ denote $\Delta \vee p$, the coarsest partition compatible with $\Delta $ 
  and $p$. Then $\Pi$ contains precisely all subsets $\hat A$ with the property that the coarsest 
  partition compatible with the restricted partitions $\Delta_{\hat A}$ and  $p_{\hat A}$ is the 
  trivial partition $\{\hat{A}\}$. With this observation in mind, let us recast the summation 
  $\sum_{\Delta\in \mathcal P(B)} \sum_{p \in \mathcal P(B)}$ as follows.
  \begin{equs}
    \E[W^B]
    & = \sum_{\Delta \in \mathcal P(B)} \sum_{p  \in \mathcal P(B)} 
        \sum_{\Pi \in \mathcal P(B)}\mathbf{1}_{\Delta \vee p = \Pi} \CJ(\Delta, p)\\
    & = \sum_{\Pi \in \mathcal P(B)} \sum_{\Delta \in \mathcal P(B)} \sum_{p \in \mathcal P(B)}
        \mathbf{1}_{\Delta \vee p = \Pi} \prod_{\hat{A}\in \Pi} \CJ(\Delta_{\hat A}, p_{\hat A})\\
    & = \sum_{\Pi\in \mathcal P(B)}\prod_{\hat A \in \Pi}  \sum_{(\Delta, p)\in \mathcal G^S(\hat A) }\CJ(\Delta, p),
  \end{equs}
  where the last equality follows from the observation made above.
  Using \eqref{cul-relation2}, this allows by inversion to identify 
  $\E_c[W_A]= \sum_{(\Delta, p)\in \mathcal G^S(A) }\CJ(\Delta, p)$.
\end{proof}
Note that in our case, the function $\CJ^\epsilon$ defined in \eqref{eq:cumulant-decomp} satisfies the
factorization property required in the Lemma above, as the integral over the product domain
$R^B_\epsilon$ factorizes over the blocks of the partition $\Delta \vee p$.

Recall that adding deterministic constants to arguments does not change cumulants of order $\ge 2$,
so that
\begin{equation}\label{eq:cumulant-barX}
  \mathbb{E}_c \bar {X}_A = \begin{cases}
    \lambda_a \epsilon^{\alpha} & \text{if } A=\{a\}   \\
    \mathbb{E}_c X_A            & \text{if } |A| >1.
  \end{cases}
\end{equation}
Observe that for $|A|=1$,  $\mathbb{E}_c \bar {X}_A  $ is a constant in time (we used the fact that
$\mathbb{E}_c [X_A]=0$, which holds as $f_{i_a} - \bar f_{i_a}$ is mean-zero). For fixed $(\Delta, p)\in {\mathcal G}^S(B)$,
$S$ the set of singletons in $\Delta$, we have that
\begin{align*}
  |\CJ^\epsilon(\Delta,p)| & \lesssim \epsilon^{|S|\alpha + \frac{|B|}{2}}\int_{R^B_\epsilon}
  \prod_{A\in \Delta \setminus S} \left|\mathbb{E}_c X_A \right| \prod_{\{a, b\}\in p}|r_a-r_b|^{2H-2} \dd r^B          \\
                       & \lesssim \epsilon^{|S|\alpha + \frac{|B|}{2}}\int_{[0, \frac{T}{\epsilon}]^B}
  \prod_{\substack{A \in \Delta \setminus S}}e^{-c \sum_{i,j \in A}|r_i-r_j|}\prod_{\{a, b\} \in p}
  |r_a - r_b|^{2H-2} \dd r^B                                                                       \\
                       & =: \tilde \CJ^\epsilon(\Delta,p)
\end{align*}
in which we used \cite[Proposition 4.8]{Hairer-Li:22} to obtain the inequality
on the second line, and, since the resulting integrand is nonnegative, we enlarged the domain of
integration from $R^B_\epsilon$ to $[0, T/\epsilon]^B$. The implicit constants absorb the
coefficients $\lambda_a, \eta_a$ and $C_H$. The exponential upper bound on $  \prod_{A\in \Delta \setminus S}
\left(\mathbb{E}_c X_A \right)$ is due to the ergodic assumption on $y$.
In the next Lemma we discuss $\tilde \CJ^\epsilon(\Delta,p)$ where  $|B| \ge 4$.

\begin{lemma}\label{lem:iterate}
  Let $ (\Delta, p) \in \mathcal{G}^S(B)$ be such that $|B| \ge 4$, $\{\sigma\} \in S_\Delta$ and
  $\rho$ be the root of $\sigma$ in $p$, i.e. $\{\sigma, \rho\} \in p$. Set $B' := B \setminus \{\sigma, \rho\}$.
  Then, denoting
  $A_\rho$ for the set in the partition $\Delta$ containing $\rho$, we define
  $ (\Delta', p') \in \mathcal{G}^S(B')$ where
  $$\ \Delta' := \{A_\rho \setminus \{\rho\}\} \cup (\Delta \setminus \{\{\sigma\}, A_\rho\})
    \text{ and } p' := p \setminus \{\{\rho, \sigma\}\},$$
  and we have
  $$\tilde \CJ^\epsilon(\Delta,p) \lesssim \epsilon^{(|S_\Delta| - |S_{\Delta'}|)(H - \f 12) + (2 - 2H)} 
    \tilde \CJ^\epsilon(\Delta', p').$$
\end{lemma}
Below is a graphical illustration of the estimate iterated three times
\begin{equation}\label{eq:iteration}
  \tilde \CJ^\epsilon\left(\begin{tikzpicture}[baseline=7.5ex,scale=0.9]
      \foreach \i in {1,...,12}
        {
          \pgfmathtruncatemacro{\x}{(\i-1)/2};
          \pgfmathtruncatemacro{\y}{\i - 2*\x};
          \coordinate (\i) at (\x,\y);
        }
      \draw[partition] (1) -- (2) -- (4) -- (6);
      \draw[partition] (5) -- (7) -- (9) -- (11) -- (12);
      \draw[partition] (3) -- (3);
      \draw[partition] (8) -- (8);
      \draw[partition] (10) -- (10);
      \foreach \i in {1,...,6}
        {
          \pgfmathtruncatemacro{\pairstart}{2*\i-1};
          \pgfmathtruncatemacro{\pairend}{2*\i};
          \draw[pairing] (\pairstart) -- (\pairend);
        }
      \fill[Red] (3) circle (2.5pt);
      \fill[Red] (8) circle (2.5pt);
      \fill[Red] (10) circle (2.5pt);
      \fill[RoyalBlue] (4) circle (2.5pt);
      \fill[RoyalBlue] (7) circle (2.5pt);
      \fill[RoyalBlue] (9) circle (2.5pt);
      \foreach \i in {1,...,12}
        {
          \coordinate[label={[font=\tiny, left]:${\i}$}] (a) at (\i);
        }
    \end{tikzpicture}\right) \lesssim
  \epsilon^{3\left(\frac{3}{2} - H\right)}
  \tilde \CJ^\epsilon\left(\begin{tikzpicture}[baseline=7.5ex,scale=0.9]
      \foreach \i in {1, ..., 6}
        {
          \pgfmathtruncatemacro{\x}{(\i-1)/2};
          \pgfmathtruncatemacro{\y}{\i - 2*\x};
          \coordinate (\i) at (\x,\y);
        }
      \draw[partition] (1) -- (2) -- (4);
      \draw[partition] (3) -- (5) -- (6);
      \foreach \i in {1, 2, 3}
        {
          \pgfmathtruncatemacro{\pairstart}{2*\i-1};
          \pgfmathtruncatemacro{\pairend}{2*\i};
          \draw[pairing] (\pairstart) -- (\pairend);
        }
      \coordinate[label={[font=\tiny, left]:$1$}] (a) at (1);
      \coordinate[label={[font=\tiny, left]:$2$}] (a) at (2);
      \coordinate[label={[font=\tiny, left]:$5$}] (a) at (3);
      \coordinate[label={[font=\tiny, left]:$6$}] (a) at (4);
      \coordinate[label={[font=\tiny, left]:$11$}] (a) at (5);
      \coordinate[label={[font=\tiny, left]:$12$}] (a) at (6);
    \end{tikzpicture}\right).
\end{equation}

\begin{proof}
  By scaling, we may assume without loss of generality that $T = 1$.
  We compute $\tilde \CJ^\epsilon(\Delta,p)$ by induction
  \begin{equation}\label{eq:J-estimate}
    \begin{split}
          & \epsilon^{-(|S_\Delta|\alpha + \frac{|B|}{2})}\tilde \CJ^\epsilon(\Delta,p)              \\
      =\  & \int_{[0, \epsilon^{-1}]^{B}} \prod_{\substack{A \in \Delta \setminus S_\Delta}}
      e^{-c \sum_{i,j \in A}|r_i - r_j|} \prod_{\{a, b\} \in p}|r_a - r_b|^{2H-2} \dd r^B          \\
      =\  & \int_{[0, \epsilon^{-1}]^{B'}} \prod_{\substack{A' \in \Delta' \setminus S_{\Delta'}}}
      e^{-c \sum_{i,j \in A'}|r_i - r_j|} \prod_{\{a, b\} \in p'}|r_a - r_b|^{2H-2}I_{\sigma}(r_{A_\rho \setminus \{\rho\}}) \dd r^{B'}
    \end{split}
  \end{equation}
  where $ I_{\sigma}(r_{A_\rho \setminus \{\rho\}})$ is given below.
  \begin{align*}
    I_{\sigma}(r_{A_\rho \setminus \{\rho\}})
     & := \int_0^{\epsilon^{-1}} e^{-2c\sum_{i \in A_\rho \setminus \{\rho\}} |r_i - r_\rho|}
    \int_0^{\epsilon^{-1}} |r_\sigma - r_\rho|^{2H-2}
    \dd r_\sigma \dd r_{\rho}                                                                              \\
     & \lesssim \int_0^{\epsilon^{-1}} e^{-c|r_a - r_\rho|}
    \left((\epsilon^{-1} - r_\rho)^{2H - 1} + r_\rho^{2H - 1}\right) \dd r_\rho                            \\
     & \lesssim \epsilon^{1 - 2H} \int_{\mathbb{R}} e^{-c|r_a - r_\rho|} \dd r_\rho
    \lesssim \epsilon^{1 - 2H},
  \end{align*}
  where in the second line we chose some $a \in A_\rho \setminus \{\rho\}$, bounded
  all other exponentials by $1$ and computed the inner integral explicitly. Such an $a$ always
  exists, as $A_\rho$ cannot be a singleton without making $(\Delta, p)$ disconnected. In the
  third step we used that $2H - 1 \in (0, 1)$, so that both $(\epsilon^{-1} - r_\rho)^{2H-1}$ and
  $r_\rho^{2H-1}$ are bounded by $\epsilon^{1-2H}$ for $r_\rho \in [0, \epsilon^{-1}]$.

  Thus $I_{\sigma}(r_{A_\rho \setminus \{\rho\}}) \lesssim \epsilon^{1 - 2H}$, uniformly over $r_{A_\rho \setminus \{\rho\}}$.
  Consequently, substituting this estimate
  into \eqref{eq:J-estimate}, we find
  \begin{align*}
               & \epsilon^{-(|S_\Delta|\alpha + \frac{|B|}{2})}\tilde \CJ^\epsilon(\Delta,p)                               \\
    \lesssim\  & \epsilon^{1 - 2H}\int_{[0, \epsilon^{-1}]^{B'}} \prod_{\substack{A' \in \Delta' \setminus S_{\Delta'}}}
    e^{-c \sum_{i,j \in A'}|r_i - r_j|} \prod_{\{a, b\} \in p'}|r_a - r_b|^{2H-2} \dd r^{B'}                             \\
    =\         & \epsilon^{1 - 2H} \epsilon^{-(|S_{\Delta'}|\alpha + \frac{|B'|}{2})}\tilde \CJ^\epsilon(\Delta', p').
  \end{align*}
  Thus, as $|B'| = |B| - 2$, we have
  $$\tilde \CJ^\epsilon(\Delta,p) \lesssim \epsilon^{(|S_\Delta| - |S_{\Delta'}|)\alpha + (2 - 2H)} \tilde \CJ^\epsilon(\Delta', p')$$
  as claimed.
\end{proof}

\begin{lemma}\label{lem:GD-kappa}
  For $ (\Delta, p) \in \mathcal{G}^S(B)$ with $S_\Delta = \varnothing$ and $|B| > 2$, we have
  \begin{equation}
    \tilde \CJ^\epsilon(\Delta,p) \lesssim \epsilon^{2 - 2H}.
  \end{equation}
\end{lemma}
We remark that we can obtain a similar estimate by a direct application of 
Proposition 4.11 and Corollary 4.13 of \cite{Hairer-Li:22}. However, for completeness
we provide an elementary proof of this estimate available in this setting in Appendix~\ref{sec:proof-GD-kappa}.

\begin{lemma}\label{lem:cumulant-conv}
  For any finite index set $B$ with $|B| > 2$, we have that
  \begin{equation}\label{eq:cumulant-lim}
    \lim_{\epsilon \to 0} \mathbb{E}_c W_B^\epsilon = 0.
  \end{equation}
  Moreover, when $|B| = 2$, the limit $\lim_{\epsilon \to 0} \mathbb{E}_c W_B^\epsilon$ exists and
  is finite. Consequently, the random vectors \eqref{eq:fdd-vector} converge in law to a centered
  Gaussian limit as $\epsilon \to 0$.
\end{lemma}
\begin{proof}
  By the cumulant estimate~\eqref{eq:cumulant-decomp}, to prove \eqref{eq:cumulant-lim} it suffices to 
  show that $\tilde \CJ^\epsilon(\Delta,p) \to 0$ for any $ (\Delta, p) \in \mathcal{G}^S(B)$. To this end, 
  we iteratively apply the estimate from Lemma~\ref{lem:iterate} to $(\Delta,p)$ until either 
  $S_{\Delta} = \varnothing$ or the underlying index set $B_\tau$ has two elements.
  Denoting $(\Delta_\tau, p_\tau) \in \mathcal{G}^S(B_\tau)$ for the terminal partition after applying
  Lemma~\ref{lem:iterate} $N$ times, we have the following three cases:
  \begin{itemize}
    \item[(i)] $S_{\Delta_\tau} = \varnothing$ and $|B_\tau| > 2$. In this case, by Lemma~\ref{lem:GD-kappa} we have the estimate
          $$\tilde \CJ^\epsilon(\Delta,p) \lesssim \epsilon^{|S_\Delta|\alpha + (2-2H)N}\tilde \CJ^\epsilon(\Delta_\tau, p_\tau)
            \lesssim \epsilon^{|S_\Delta|\alpha + (2-2H)(N+1)}.$$
          Thus, as the exponent is strictly positive, the integral vanishes as $\epsilon \to 0$.
    \item[(ii)] $S_{\Delta_\tau} = \varnothing$, $|B_\tau| = |\{a, b\}| = 2$ and $\Delta_\tau = p_\tau = \{B_\tau\}$. Namely, we have that
          $(\Delta_\tau, p_\tau) = 
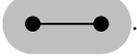
\begin{tikzpicture}[baseline=-0.5ex, scale=0.9]
              \coordinate (1) at (0,0);
              \coordinate (2) at (1,0);
              \draw[partition] (1) -- (2);
              \draw[pairing] (1) -- (2);
            \end{tikzpicture}$.
          \begin{align*}
            \tilde \CJ^\epsilon(\{B_\tau\}, \{B_\tau\}) & = \epsilon \int_{0}^{T/\epsilon}\!\int_{0}^{T/\epsilon}
            e^{-c|r_a - r_b|} |r_a-r_b|^{2H-2} \dd r_a \dd r_b                                                        \\
                                                      & \lesssim T \int_\mathbb{R} e^{-c|u|} |u|^{2H-2} \dd u < \infty.
          \end{align*}
          In this case, $N$ is necessarily $\frac{|B|}{2} - 1 > 0$ and
          $$\tilde \CJ^\epsilon(\Delta,p) \lesssim \epsilon^{|S_\Delta|\alpha 
            + (2-2H)\left(\frac{|B|}{2}-1\right)} \tilde \CJ^\epsilon(\Delta_\tau, p_\tau) \lesssim
            \epsilon^{|S_\Delta|\alpha + (2-2H)\left(\frac{|B|}{2}-1\right)}.$$
          Hence, as the exponent is strictly positive,
          the left hand side vanishes as $\epsilon \to 0$.
    \item[(iii)] $S_{\Delta_\tau} \neq \varnothing$, $|B_\tau| = |\{a, b\}| = 2$ and $p_\tau = \{B_\tau\}$ 
          and $\Delta_\tau=\{\{a\}, \{b\}\}$. Namely, we have
          $(\Delta_\tau, p_\tau) = \begin{tikzpicture}[baseline=-0.5ex, scale=0.9]
              \coordinate (1) at (0,0);
              \coordinate (2) at (1,0);
              \draw[partition] (1) -- (1);
              \draw[partition] (2) -- (2);
              \draw[pairing] (1) -- (2);
            \end{tikzpicture}$.
          Again, $N = \frac{|B|}{2} - 1$ and we may compute $\tilde \CJ^\epsilon(\Delta_\tau, p_\tau)$ directly as
          $$\tilde \CJ^\epsilon( \{\{a\},\{b\}\}, B_\tau) = \epsilon^{2\alpha +1}\int_0^{T/\epsilon}
            \int_0^{T/\epsilon}  |u-v|^{2H-2}\dd u \dd v \simeq \epsilon^{2\alpha + 1 - 2H} T^{2H} = T^{2H}.$$
          Thus,
          $$\tilde \CJ^\epsilon(\Delta,p) \lesssim \epsilon^{(|S_\Delta|-2)\alpha 
            + (2-2H)\left(\frac{|B|}{2}-1\right)}\tilde \CJ^\epsilon(\Delta_\tau, p_\tau)
            \lesssim \epsilon^{(|S_\Delta|-2)\alpha + (2-2H)\left(\frac{|B|}{2}-1\right)}$$
          which vanishes as $\epsilon \to 0$.
  \end{itemize}
  Thus, all the cumulants of $W^\epsilon_B$ with $|B| > 2$ vanish as $\epsilon \to 0$.

  In the case where $|B| = 2$, we have that
  \begin{equation*}
    \mathbb{E}_c W_B^\epsilon = \CJ^\epsilon\left( \begin{tikzpicture}[baseline=-0.5ex, scale=0.9]
        \coordinate (1) at (0,0);
        \coordinate (2) at (1,0);
        \draw[partition] (1) -- (2);
        \draw[pairing] (1) -- (2);
      \end{tikzpicture} \right) + \CJ^\epsilon\left( \begin{tikzpicture}[baseline=-0.5ex, scale=0.9]
        \coordinate (1) at (0,0);
        \coordinate (2) at (1,0);
        \draw[partition] (1) -- (1);
        \draw[partition] (2) -- (2);
        \draw[pairing] (1) -- (2);
      \end{tikzpicture} \right).
  \end{equation*}
  Writing $B = \{a, b\}$, the first term is the covariance of the $V^\epsilon$-parts,
  $$\CJ^\epsilon\left( \begin{tikzpicture}[baseline=-0.5ex, scale=0.9]
        \coordinate (1) at (0,0);
        \coordinate (2) at (1,0);
        \draw[partition] (1) -- (2);
        \draw[pairing] (1) -- (2);
      \end{tikzpicture} \right)
    = \eta_a \eta_b \E\left[V^\epsilon_{s_a, t_a}(x_a)_{i_a} V^\epsilon_{s_b, t_b}(x_b)_{i_b}\right],$$
  which converges, as $\epsilon \to 0$, to
  $$\eta_a \eta_b \left|[s_a, t_a] \cap [s_b, t_b]\right|
    \big(\Sigma(x_a, x_b) + \Sigma(x_b, x_a)^\top\big)_{i_a i_b}$$
  by Proposition~\ref{prop:V-eps-moments}. The
  second term can be computed directly as
  \begin{align*} \CJ^\epsilon\left( \begin{tikzpicture}[baseline=-0.5ex, scale=0.9]
                                                     \coordinate (1) at (0,0);
                                                     \coordinate (2) at (1,0);
                                                     \draw[partition] (1) -- (1);
                                                     \draw[partition] (2) -- (2);
                                                     \draw[pairing] (1) -- (2);
                                                   \end{tikzpicture} \right) & = C_H \lambda_a \lambda_b \epsilon^{2\alpha + 1}
               \int_{s_a/\epsilon}^{t_a/\epsilon}\int_{s_b/\epsilon}^{t_b/\epsilon}
               |r_a - r_b|^{2H - 2} \dd r_a \dd r_b                                                                                                                    \\
                                                                                & = C_H \lambda_a \lambda_b
               \int_{s_a}^{t_a}\int_{s_b}^{t_b}
               |r_a - r_b|^{2H - 2} \dd r_a \dd r_b
               = \lambda_a \lambda_b \E\left[B^H_{s_a, t_a} B^H_{s_b, t_b}\right],
  \end{align*}
  an expression independent of $\epsilon$. This shows the existence of the limit of the second
  order cumulants. Note that, in accordance with Lemma~\ref{lem:cross-cov}, no cross term between
  the $B^H$ and $V^\epsilon$ parts appears.

  Finally, by the computations above, all the joint cumulants of the entries of
  \eqref{eq:fdd-vector} of order strictly greater than two vanish as $\epsilon \to 0$, while the
  second order cumulants converge. Hence, the random vectors \eqref{eq:fdd-vector} converge in law 
  to this Gaussian limit.
\end{proof}

\section{Proof of the main theorem}\label{sec:proof-main}

\subsection{The fluctuation process as a Young differential equation}\label{sec:z-eps-YDE}

In this subsection we make rigorous the derivation sketched in Step \textbf{A} of
Section~\ref{sec:proof-strat} in the language of nonlinear Young integrals. 
In particular, we justify that the dynamics of \(x^\epsilon\) given in \eqref{main-equation} (originally constructed in 
\cite{Hairer-Li:20} via the stochastic sewing Lemma) can be interpreted as a nonlinear Young differential equation.
Building on this, we then derive the linear YDE satisfied by the fluctuation process
$z^\epsilon$. 

Throughout this section we fix H\"older exponents 
$\alpha' \in \left(\f 12, H\right)$ and $\beta \in \left(1 - \alpha', \f 12\right)$
which is possible since $H > \f 12$. In particular $\alpha' + \beta > 1$.

We define 
\begin{equation}\label{eq:U-eps-field}
  U^\epsilon_t(x) := \int_0^t f(x, y^\epsilon_r) \dd B^H_r
  = \bar f(x) B^H_t + \epsilon^{H - \f 12} V^\epsilon_t(x).
\end{equation}
Since $\bar f \in C^3_b$ and $B^H \in C^{\alpha'}([0,T]; \R^m)$ a.s. while $V^\epsilon
\in C^{\beta}([0,T]; C^3_b)$ a.s. for each fixed $\epsilon$, we have $U^\epsilon \in
C^{\beta}([0,T]; C^3_b)$ a.s. Similarly $G^\epsilon \in C^\beta([0,T]; C^3_b)$ a.s.
where $G^\epsilon$ is the drift fluctuation field from \eqref{eq:G-eps-def-intro}
(cf. Lemma~\ref{lem:G-eps-vanish}). Thus, the nonlinear Young integrals
$\int_0^\cdot U^\epsilon(\dd s, x^\epsilon_s)$, $\int_0^\cdot V^\epsilon(\dd s, x^\epsilon_s)$ and
$\int_0^\cdot G^\epsilon(\dd s, x^\epsilon_s)$
are well defined by Lemma~\ref{lem:NY-int}.

\begin{lemma}\label{lem:z-eps-NYDE}
  Let $\epsilon > 0$ be fixed. The solution $x^\epsilon$ of \eqref{main-equation}
  satisfies the nonlinear YDE
  \begin{equation}\label{eq:x-eps-NYDE}
    x^\epsilon_t = x_0 + \int_0^t U^\epsilon(\dd s, x^\epsilon_s) + \int_0^t g(x^\epsilon_s, y^\epsilon_s) \dd s.
  \end{equation}
  Consequently, the processes $w^\epsilon$ and
  $\tilde w^\epsilon$ defined in \eqref{eq:V-epsilon-def} satisfy
  \begin{equation}\label{eq:w-eps-identification}
    w^\epsilon = \int_0^\cdot V^\epsilon(\dd s, x^\epsilon_s) \in C^\beta([0,T]; \R^d),
    \qquad
    \tilde w^\epsilon = \int_0^\cdot G^\epsilon(\dd s, x^\epsilon_s) \in C^\beta([0,T]; \R^d),
  \end{equation}
  and the fluctuation process $z^\epsilon = \epsilon^{\f 12 - H}(x^\epsilon - \bar x)$ satisfies
  the linear YDE
  \begin{equation}\label{eq:z-eps-YDE-rigorous}
    z^\epsilon_t = w^\epsilon_t + \tilde w^\epsilon_t + \int_0^t A^\epsilon_s z^\epsilon_s \dd B^H_s
    + \int_0^t \tilde A^\epsilon_s z^\epsilon_s \dd s
  \end{equation}
  with
  \begin{equation}\label{eq:A-eps-defs}    A^\epsilon_t = \int_0^1 D\bar f(\theta x^\epsilon_t + (1 - \theta) \bar x_t) \dd \theta, \qquad
    \tilde A^\epsilon_t = \int_0^1 D\bar g(\theta x^\epsilon_t + (1 - \theta) \bar x_t) \dd \theta,
  \end{equation}
  where the integral against $B^H$ is a classical Young integral and all equalities above hold a.s.
\end{lemma}
\begin{proof}
The integral $\int_0^t f(x^\epsilon_s, y^\epsilon_s) \dd B^H_s$, as appearing in
\eqref{main-equation} and \cite[Lemma 3.10]{Hairer-Li:20} is constructed 
via an application of the stochastic sewing Lemma to 
$$\Xi_{s,t} := \int_s^t f(x^\epsilon_s, y^\epsilon_r) \dd B^H_r = U^\epsilon_{s,t}(x^\epsilon_s).$$
In particular,
by \cite[Equation (3.10)]{Hairer-Li:20}, for every fixed $(s, t) \in \Delta_T$ the integral
$\int_s^t f(x^\epsilon_r, y^\epsilon_r) \dd B^H_r$ is the limit in $L^p$ of the Riemann sums
$\sum_{[u, v] \in \pi} U^\epsilon_{u,v}(x^\epsilon_u)$ along partitions $\pi$ of $[s, t]$ with
vanishing mesh\footnote{Note that this construction via stochastic sewing in \cite{Hairer-Li:20} was 
for the applications therein in which $y$ is not assumed to be independent of $B^H$.}. On the other hand, 
since almost surely $U^\epsilon \in C^\beta([0,T]; C^3_b)$ and
$x^\epsilon \in C^{\alpha'}$ with $\alpha' + \beta > 1$, Lemma~\ref{lem:NY-int} shows that the
same Riemann sums converge almost surely to the nonlinear Young integral
$\int_s^t U^\epsilon(\dd r, x^\epsilon_r)$. Since the $L^p$-limit and the almost sure limit of a
single sequence of random variables coincide, the two integrals agree almost
surely for each fixed $(s, t)$. As both are continuous in $(s, t)$, they agree almost surely
simultaneously for all $(s, t) \in \Delta_T$. This proves \eqref{eq:x-eps-NYDE}.

On the other hand, by the linearity of the nonlinear Young integral in the integrand field, the decomposition
\eqref{eq:U-eps-field}, and the consistency statement of Lemma~\ref{lem:NY-int}
applied to the field $(t, x) \mapsto \bar f(x) B^H_t$, we obtain
$$\int_0^t U^\epsilon(\dd s, x^\epsilon_s)
= \int_0^t \bar f(x^\epsilon_s) \dd B^H_s + \epsilon^{H - \f 12} \int_0^t V^\epsilon(\dd s, x^\epsilon_s),$$
where the first integral on the right-hand side is a classical Young integral. Combining this
with \eqref{eq:x-eps-NYDE} and the definition \eqref{eq:V-epsilon-def} of $w^\epsilon$ yields
the first identity in \eqref{eq:w-eps-identification}. The regularity $w^\epsilon \in C^\beta$
follows from Lemma~\ref{lem:NY-int} since $\alpha' + \beta > 1$. The drift term $\tilde w^\epsilon$ 
can be treated similarly.

Finally, subtracting from
\eqref{eq:x-eps-NYDE} the equation
$\bar x_t = x_0 + \int_0^t \bar f(\bar x_s) \dd B^H_s + \int_0^t \bar g(\bar x_s) \dd s$ for
the effective limit, multiplying by $\epsilon^{\f 12 - H}$ and Taylor expanding
$\bar f(\bar x_s)$ and $\bar g(\bar x_s)$ around $x^\epsilon_s$ as in
Section~\ref{sec:proof-strat} yields \eqref{eq:z-eps-YDE-rigorous}.
\end{proof}

\begin{remark}
  We remark that the integrand $s \mapsto f(x^\epsilon_s, y^\epsilon_s)$ is in general merely
  continuous in time, owing to the assumptions on the fast process, so that the integral in
  \eqref{main-equation} cannot directly be interpreted as a classical Young integral. The nonlinear Young
  formulation \eqref{eq:x-eps-NYDE} resolves this by integrating the field $U^\epsilon$, which is
  regular in time thanks to the stochastic cancellations encoded in the Wiener integrals, along
  the path $x^\epsilon$.
\end{remark}

\subsection{A linear Residue Lemma for controlled YDE}

We provide the following Gr\"onwall type Lemma from which we deduce the aforementioned Residue Lemma.
We note that similar estimates have been studied in \cite[Section 6.2]{Lejay:10}. However, the estimates therein 
have slightly different regularity assumptions and are not directly applicable to our setting (namely there, 
the perturbation is assumed to have the same regularity as the solution). 
We provide a self-contained proof of the following Lemma for completeness.

\begin{lemma}\label{lem:gronwall}
  Let $\CU, \CV$ be Banach spaces and suppose $X \in C^\alpha([0, T], \CU)$, $h \in C^\beta([0, T], \CV)$ and
  $A \in C^\gamma([0, T], \mathbb{L}(\CV, \mathbb{L}(\CU, \CV)))$ for some
  $\alpha \in (\frac{1}{2}, 1]$ and $0 < \beta \le \gamma < 1$ with $\beta \le \alpha$ and
  $\alpha + \beta > 1$. Suppose that $Y : [0, T] \to \CV$
  satisfies the controlled YDE
  \begin{equation}\label{eq:gronwall-YDE}
    Y_t = Y_0 + \int_0^t A_s Y_s \dd X_s + h_{0,t}.
  \end{equation}
  Then,
  \begin{equation}\label{eq:infty-norm-est}
    \|Y\|_{\infty; [0, T]} \le C_1 e^{C_2(\|A\|_\gamma^{\frac{1}{\gamma}} + \|X\|_\alpha^{\frac{1}{\alpha}})}
    \left(\|Y_0\| + \|h\|_\beta\right)
  \end{equation}
  and
  \begin{equation}\label{eq:beta-norm-est}
    \|Y\|_{\beta; [0, T]} \le C_3 \left(1 + \|A\|_\gamma^{\frac{1}{\gamma}} + \|X\|_\alpha^{\frac{1}{\alpha}}\right)
    ((\|X\|_\alpha + \|A\|_\gamma) \|Y\|_{\infty; [0, T]} + \|h\|_\beta)
  \end{equation}
  for some constants $C_1, C_2, C_3 > 0$ depending on $\alpha, \beta, \gamma, T$ and $\|A\|_\infty$.
\end{lemma}
\begin{remark}
  We remark that, by applying the estimate~\eqref{eq:infty-norm-est} within \eqref{eq:beta-norm-est}
  and by enlarging the constants $C_1$ and $C_2$ within the estimate~\eqref{eq:infty-norm-est}, the terms
  which depend polynomially on $\|X\|_\alpha$ and $\|A\|_\gamma$ can be absorbed into an exponential.
  Namely, we obtain an estimate of the form
  \begin{equation}\label{eq:gronwal-rmk}
    \|Y\|_{\beta; [0, T]} \le C_1' e^{C_2'(\|A\|_\gamma^{\frac{1}{\gamma}} + \|X\|_\alpha^{\frac{1}{\alpha}})}
    \left(\|Y_0\| + \|h\|_\beta\right).
  \end{equation}
\end{remark}
\begin{proof}
 Using triangle and Young's inequality, for any $s < t \in [0, T]$ with $|t - s| \le 1$, we obtain the bound
  \begin{equation}
    \begin{split}
      \|Y_{s, t}\| & \le \|h_{s, t}\| + \|A_s Y_s X_{s, t}\| + \left\|\int_s^t A_r Y_r \dd X_r - A_s Y_s X_{s, t}\right\| \\
                   & \le \|h\|_\beta |t - s|^\beta + \|A\|_\infty \|Y\|_{\infty; [s, t]} \|X\|_\alpha|t - s|^\alpha +
      C\|A_\cdot Y_\cdot\|_{\beta; [s, t]} \|X\|_\alpha |t - s|^{\beta + \alpha}
    \end{split}
  \end{equation}
  Now, as
  \begin{equation}\label{eq:AY-beta-est}
    \|A_\cdot Y_\cdot\|_{\beta; [s, t]} \le \|A\|_{\gamma} \|Y\|_{\infty; [s, t]} |t - s|^{\gamma - \beta}
     + \|A\|_\infty \|Y\|_{\beta; [s, t]},
  \end{equation}
  by taking
  \begin{equation}\label{small-1}
  |t - s|^\alpha \le (2C\|A\|_\infty\|X\|_\alpha)^{-1},
  \end{equation}
   we have that
  \begin{equation}\label{eq:Y-young-ineq}
    \|Y\|_{\beta; [s, t]} \le 2\|h\|_\beta + \left(2\|A\|_\infty \|X\|_\alpha|t - s|^{\alpha - \beta} +
    \|A\|_\infty^{-1} \|A\|_{\gamma}|t - s|^{\gamma - \beta} \right) \|Y\|_{\infty; [s, t]}.
  \end{equation}
We now obtain the supremum bound, by applying (\ref{eq:Y-young-ineq}) in the second step
  \begin{align*}
          & \|Y\|_{\infty; [s, t]} \le \|Y_s\| + \|Y\|_{\beta; [s, t]}|t - s|^\beta                \\
    \le\  & \|Y_s\| + 2\|h\|_\beta|t - s|^\beta + \left(2\|A\|_\infty \|X\|_\alpha|t - s|^{\alpha} +
    \|A\|_\infty^{-1} \|A\|_{\gamma} |t - s|^\gamma\right) \|Y\|_{\infty; [s, t]}                    \\
    \le\  & \|Y_s\| + 2\|h\|_\beta|t - s|^\beta
    + \left(C^{-1} + \|A\|_\infty^{-1} \|A\|_{\gamma} |t - s|^\gamma\right) \|Y\|_{\infty; [s, t]}
  \end{align*}
  In the last step, we used (\ref{small-1}).
  Consequently, taking $|t - s|$ sufficiently small so that
  $$|t - s|^{\gamma} \le \left(\frac{1}{2} - C^{-1}\right)\|A\|_\infty \|A\|_\gamma^{-1} =: c_3 \|A\|_\gamma^{-1},$$
  (where we note that $\frac{1}{2} - C^{-1} > 0$ as tracking the constant in the Young's inequality
  $C = 2^{\alpha + \beta} \zeta(\alpha + \beta) > 2$),
  we obtain the estimate (recalling that we had assumed that $|t - s| < 1$),
  $$\|Y\|_{\infty; [s, t]} \le 2\|Y_s\| + 4\|h\|_\beta.$$
 We take a uniform partition $(t_n)_n$ of $[0, T]$ with mesh size
  $$|t_{n + 1} - t_n| = 1 \wedge (c_3^{-1} \|A\|_\gamma)^{-\frac{1}{\gamma}} \wedge (2C\|A\|_\infty \|X\|_\alpha)^{-\frac{1}{\alpha}}$$
  so that there are at most
  $(2C\|A\|_\infty)^{\frac{1}{\alpha}} T \|X\|_\alpha^{\frac{1}{\alpha}}
    + {c_3}^{-\frac{1}{\gamma}} T \|A\|_\gamma^{\frac{1}{\gamma}} + 1 =:
    c_1 \|X\|_\alpha^{\frac{1}{\alpha}} + c_2\|A\|_\gamma^{\frac{1}{\gamma}} + 1$
  number of intervals in the partition. Iterating the above estimate, we obtain
  \begin{align*}
    \|Y\|_\infty & \lesssim 2^{c_1 \|X\|_\alpha^{\frac{1}{\alpha}} + c_2\|A\|_{\gamma}^{\frac{1}{\gamma}}}\|Y_0\|
    + 4 \|h\|_\beta
    \left(1 + 2 + 2^2 + \cdots + 2^{c_1\|X\|_\alpha^{\frac{1}{\alpha}} + c_2\|A\|_\gamma^{\frac{1}{\gamma}} + 1}\right)                      \\
                 & \lesssim 2^{c_1\|X\|_\alpha^{\frac{1}{\alpha}} + c_2\|A\|_{\gamma}^{\frac{1}{\gamma}}}\left(\|Y_0\| + \|h\|_\beta\right),
  \end{align*}
  proving \eqref{eq:infty-norm-est}.

  Finally, summing \eqref{eq:Y-young-ineq} over the same partition, whose number of intervals is bounded by
  $c_1 \|X\|_\alpha^{1/\alpha} + c_2 \|A\|_\gamma^{1/\gamma} + 1$, and leveraging the subadditivity
  of the H\"older seminorm in terms of the intervals
  (noting that we had assumed $|t - s| \le 1$, and that $\beta \le \alpha \wedge \gamma$ so that
  the bracket in \eqref{eq:Y-young-ineq} is bounded on each such interval), we obtain the
  estimate~\eqref{eq:beta-norm-est}.
\end{proof}

\begin{lemma}\label{lem:residue-yde}
  Taking $\alpha, \beta, \gamma, \CU, \CV$ as in Lemma~\ref{lem:gronwall},
  suppose $X, \tilde X \in C^\alpha([0, T], \CU)$, $Z, \tilde Z \in C^\beta([0, T], \CV)$, 
  $A, \tilde A \in C^\gamma([0, T], \mathbb{L}(\CV, \mathbb{L}(\CU, \CV)))$ such that \(A_0 = \tilde A_0\) 
  and $z, \tilde z : [0, T] \to \CV$ satisfying the controlled YDE
  \begin{equation}
    z_t = \int_0^t A_s z_s \dd X_s + Z_t, \quad \tilde z_t = \int_0^t \tilde A_s \tilde z_s \dd \tilde X_s + \tilde Z_t
  \end{equation}
  respectively. Then, denoting $\Delta g = g - \tilde g$ for any function $g$, we have
  $$\|z - \tilde z\|_\beta \le C_1
    e^{C_2 (\|A\|_\gamma^{\frac{1}{\gamma}} + \|X\|_\alpha^{\frac{1}{\alpha}}
        + \|\tilde A\|_\gamma^{\frac{1}{\gamma}} + \|\tilde X\|_\alpha^{\frac{1}{\alpha}})}
    \left(\|\Delta Z_0\| + \|\Delta Z\|_\beta + \|\Delta X\|_\alpha
    + \|\Delta A\|_\gamma\right)$$
  where $C_1, C_2 > 0$ are constants depending on $\alpha, \beta, \gamma, T,
    \|\tilde Z_0\|, \|\tilde Z\|_\beta, \|A\|_\infty$ and $\|\tilde A\|_\infty$.
\end{lemma}
We remark that, tracking the constant carefully, one observes that the constant $C_1$ depends only
linearly on $\|\tilde Z_0\|, \|\tilde Z\|_\beta$ while $C_2$ does not depend on
$\|\tilde Z_0\|, \|\tilde Z\|_\beta$ at all. This is useful should we wish to obtain a moment
bound on $z - \tilde z$.
\begin{proof}
  Setting $Y_t := z_t - \tilde z_t$, we have that
  $$\dd Y_t = A_t Y_t \dd X_t + \dd (Z - \tilde Z)_t - \tilde A_t \tilde z_t \dd \tilde X_t + A_t \tilde z_t \dd X_t.$$
  Thus, applying Lemma~\ref{lem:gronwall} to the above equation with the forcing
  $$h_t := Z_t - \tilde Z_t - \int_0^t \tilde A_s \tilde z_s \dd \tilde X_s + \int_0^t A_s \tilde z_s \dd X_s,$$
  we obtain that
  \begin{equation}\label{eq:delta-z-est}
    \|z - \tilde z\|_\beta \lesssim e^{C (\|A\|_\gamma^{\frac{1}{\gamma}} + \|X\|_\alpha^{\frac{1}{\alpha}})}
    \left(\|\Delta Z_0\| + \|h\|_\beta\right),
  \end{equation}
  and it remains to estimate $\|h\|_\beta$. Applying Young's inequality, we obtain for any
  $s < t \in [0, T]$,
  \begin{align*}
    \|h_{s, t}\| \le \|\Delta Z_{s, t}\| &
    + \left\|\int_s^t \tilde A_r \tilde z_r \dd (\tilde X - X)_r\right\|
    + \left\|\int_s^t (\tilde A_r - A_r) \tilde z_r \dd X_r\right\|                                                 \\
    \lesssim \|\Delta Z_{s, t}\|         &
    + \|\tilde A_{\cdot} \tilde z_{\cdot}\|_\infty\|\tilde X_{s, t} - X_{s, t}\|
    + \|\tilde A_{\cdot} \tilde z_{\cdot}\|_\beta \|\tilde X - X\|_\alpha |t - s|^{\alpha + \beta}                    \\
                                         & + \|(\tilde A_{\cdot} - A_{\cdot}) \tilde z_{\cdot}\|_\infty\|X_{s, t}\|
    + \|(\tilde A_{\cdot} - A_{\cdot}) \tilde z_{\cdot}\|_\beta \|X\|_\alpha |t - s|^{\alpha + \beta}.
  \end{align*}
  Thus, we find
  \begin{align*}
    \|h\|_{\beta} \le \|\Delta Z\|_\beta
     & + \|\tilde A\|_\infty \|\tilde z\|_\infty \|\Delta X\|_\alpha T^{\alpha - \beta}
    + \|\tilde A_{\cdot} \tilde z_{\cdot}\|_\beta \|\Delta X\|_\alpha T^\alpha               \\
     & + \|\Delta A\|_\infty \|\tilde z\|_\infty \|X\|_\alpha T^{\alpha - \beta}
    + \|\Delta A_{\cdot} \tilde z_{\cdot}\|_\beta \|X\|_\alpha T^\alpha                      \\
    \lesssim \|\Delta Z\|_\beta
     & + (\|\tilde A\|_\infty \|\tilde z\|_\infty + \|\tilde A\|_\gamma \|\tilde z\|_\infty
    + \|\tilde A\|_\infty \|\tilde z\|_\beta) \|\Delta X\|_\alpha                           \\
     & + (\|\Delta A\|_\infty \|\tilde z\|_\infty + \|\Delta A\|_\gamma \|\tilde z\|_\infty
    + \|\Delta A\|_\infty \|\tilde z\|_\beta) \|X\|_\alpha.
  \end{align*}
  Now, bounding $\|\tilde z\|_\infty \vee \|\tilde z\|_\beta$ by Lemma~\ref{lem:gronwall} (and
  its subsequent remark \eqref{eq:gronwal-rmk}), and by absorbing relevant terms of
  polynomial order involving $\|\tilde A\|_\gamma$ and $\|\tilde X\|_\alpha$ into the exponential,
  we obtain the estimate
  \begin{align*}
    \|h\|_{\beta} & \lesssim \|\Delta Z\|_\beta +
    (\|\Delta X\|_\alpha
    + \|\Delta A\|_\gamma)
    e^{C'(\|\tilde A\|_\gamma^{\frac{1}{\gamma}} + \|\tilde X\|_\alpha^{\frac{1}{\alpha}})}
    \left(\|\tilde Z_0\| + \|\tilde Z\|_\beta\right)
  \end{align*}
  for some constant $0 < C'$.
  Consequently, substituting the above into \eqref{eq:delta-z-est}, we obtain the desired inequality.
\end{proof}

\subsection{Proof of Theorem~\ref{thm:main}}\label{subsec:proof-main-theorem}

In this section we complete the proof of Theorem~\ref{thm:main}. Recall the exponents
$\alpha' \in (\f 12, H)$ and $\beta \in (1 - \alpha', \f 12)$ fixed in
Section~\ref{sec:z-eps-YDE}. By Lemma~\ref{lem:z-eps-NYDE}, $z^\epsilon$ solves the linear
YDE~\eqref{eq:z-eps-YDE-rigorous}, which we shall compare with the limiting
equation~\eqref{eq:z-lim-eq}. We first settle the issue of well-posedness for the limiting
equation by a simple fixed point argument. We again note \cite{Lejay:10} for similar well-posedness results 
with slightly different regularity assumptions.

\begin{lemma}\label{lem:well-posed}
  Let $T > 0$ and let $\CU, \CV$ be Banach spaces, $\alpha \in (\frac{1}{2}, 1]$ and
  $0 < \beta \le \gamma < 1$ with $\beta \le \alpha$ and $\alpha + \beta > 1$. Then, taking
  $$X \in C^\alpha([0, T], \CU), \qquad h \in C^\beta([0, T], \CV), 
    \qquad A \in C^\gamma([0, T], \mathbb{L}(\CV, \mathbb{L}(\CU, \CV)))$$
  the controlled YDE
  \begin{equation}\label{eq:z-well-posed}
    z_t = z_0 + \int_0^t A_s z_s \dd X_s + h_{0,t}
  \end{equation}
  has a unique solution $z \in C^{\beta}([0, T], \CV)$.
\end{lemma}
\begin{proof}
  Defining the map $\Psi : C^\beta([0, t], \CV) \to C^\beta([0, t], \CV)$ by
  $$\Psi(z)_t := z_0 + \int_0^t A_s z_s \dd X_s + h_{0,t},$$
  we compute for $z, z' \in C^\beta([0, T], \CV)$ that
  $$\|\Psi(z) - \Psi(z')\|_\beta \lesssim \|X\|_\alpha \|z - z'\|_\beta (t^{\alpha - \beta} + t^{\alpha + \gamma - \beta}).$$
  Hence, by taking $t$ small enough, we can ensure that $\Psi$ is a strict contraction on $C^\beta([0, t], \CV)$.
  Thus, we deduce the existence of a unique solution to \eqref{eq:z-well-posed} on the full interval $[0,T]$ by a simple bootstrapping argument.
\end{proof}

\begin{remark}\label{rem:weak-uniqueness}
  We remark that the above proof demonstrates that a Picard iteration converges to the solution of
  \eqref{eq:z-well-posed}. Consequently, in the case where $h, A$ and $X$ are random, we also obtain
  weak uniqueness of the solutions. Indeed, at each step of the iteration, the Picard iterates have
  the same laws.
\end{remark}

In order to treat the Young and the Riemann integrals in \eqref{eq:z-eps-YDE-rigorous} on the
same footing, we collect the two drivers into the path
$$X^H_t := (B^H_t, t) \in \R^{m + 1},$$
which belongs to $C^{\alpha'}([0,T]; \R^{m+1})$ almost surely, and the two coefficients into
$$\mathbf{A}^\epsilon := (A^\epsilon, \tilde A^\epsilon) \in
\left\{\mathbf{A} \in C^\beta\big([0,T]; \L(\R^d, \L(\R^{m+1}, \R^d))\big) : 
  \mathbf{A}_0 = (D\bar{f}(x_0), D\bar{g}(x_0))\right\} =: \mathcal{A}_0,$$ 
so that~\eqref{eq:z-eps-YDE-rigorous} reads
$z^\epsilon_t = w^\epsilon_t + \tilde w^\epsilon_t + \int_0^t \mathbf{A}^\epsilon_s z^\epsilon_s \dd X^H_s$.

In view of Lemma~\ref{lem:well-posed}, applied with the driver $X \in C^{\alpha'}([0,T]; \R^{m+1})$,
coefficient $\mathbf{A} \in C^\beta$ and forcing $w \in C^\beta$, we may define the solution map
\begin{equation}\label{eq:solution-map}
  \Phi : C^\beta([0,T]; \R^d) \times \mathcal{A}_0 \times C^{\alpha'}([0,T]; \R^{m+1})
  \to C^\beta([0,T]; \R^d),
\end{equation}
which maps $(w, \mathbf{A}, X)$ to the unique solution $z$ of the linear YDE
\begin{equation}\label{eq:linear-YDE-generic}
  z_t = w_t + \int_0^t \mathbf{A}_s z_s \dd X_s.
\end{equation}
The residue Lemma~\ref{lem:residue-yde}, applied with, in the notation of that Lemma, the
exponents $(\alpha, \beta, \gamma)$ taken to be $(\alpha', \beta, \beta)$, shows that $\Phi$ is
continuous. By Lemma~\ref{lem:z-eps-NYDE} we have,
almost surely,
\begin{equation}\label{eq:z-eps-as-Phi}
  z^\epsilon = \Phi\left(w^\epsilon + \tilde w^\epsilon, (A^\epsilon, \tilde A^\epsilon), X^H\right),
\end{equation}
while the solution of the limiting equation \eqref{eq:z-lim-eq} is
$z = \Phi\big(\int_0^\cdot V(\dd s, \bar x_s), (D\bar f(\bar x), D \bar g(\bar x)), X^H\big)$;
here the nonlinear
Young integral is well defined by Lemma~\ref{lem:NY-int} since,
$V \in C^\beta([0,T]; C^3_b)$ almost surely (recall Proposition~\ref{prop:V-eps-moments}).
Thus, by the continuous mapping theorem, it remains to justify the joint convergence in law of the inputs of
\eqref{eq:z-eps-as-Phi}. 

We first record two auxiliary convergence results.

\begin{lemma}\label{lem:DFn-cvg}
  Taking $\varphi \in C^2_b$, $\beta < \alpha \in (0, 1)$ and a sequence $(x^n)$
  converging to $x$ in $C^\alpha([0, T])$, we have that
  \begin{equation}
    \lim_{n \to \infty} \|D\varphi(x^n) - D\varphi(x)\|_\beta = 0,
  \end{equation}
  where the convergence is in $C^\beta([0,T]; \L(\R^d, \R^d))$. 
\end{lemma}
\begin{proof}
  As we are now working with finite dimensional spaces, the embedding $C^\alpha \hookrightarrow C^\beta$ 
  is compact, and any bounded subset of $C^\alpha([0, T])$ is pre-compact in $C^\beta([0, T])$. It suffices to show
  $\{D\varphi(x^n) - D\varphi(x)\}_{n}$ is bounded in $C^\alpha$ and $D\varphi(x^n) - D\varphi(x)$ converges in the supremum norm.
  The bound in $C^\alpha$ is immediate.  For any $s < t \in [0, T]$, we have
  \begin{equation}
    \begin{split}
      \|(D\varphi(x^n) - D\varphi(x))_{s, t}\| & \le \|D\varphi(x^n_t) - D\varphi(x^n_s)\| + \|D\varphi(x_t) - D\varphi(x_s)\|               \\
                                   & \le \|D^2\varphi\|_\infty (\|x^n\|_\alpha + \|x\|_\alpha) |t - s|^\alpha.
    \end{split}
  \end{equation}
  Since  $x^n$ converges to $x$ in $C^\alpha$, the H\"older norm is uniformly bounded in $n$
  \begin{equation}\label{eq:DF-bound}
    \|D\varphi(x^n) - D\varphi(x)\|_\alpha \le \|D^2\varphi\|_\infty (\|x^n\|_\alpha + \|x\|_\alpha).
  \end{equation}
  For the convergence, note that
  $$\|D\varphi(x^n) - D\varphi(x)\|_\infty \le \sup_{t \in [0, T]} \|D^2\varphi\|_\infty \|x^n_t - x_t\|
    \le \|D^2\varphi\|_\infty (\|x^n_0 - x_0\| + \|x^n - x\|_\alpha T^\alpha) \to 0$$
  as $n \to \infty$, completing the proof.
\end{proof}

\begin{lemma}\label{lem:A-eps-conv}
  Let $A := D \bar f(\bar x)$, $\tilde A := D \bar g(\bar x)$ and let $A^\epsilon, \tilde A^\epsilon$
  be as in \eqref{eq:A-eps-defs}. Then, as
  $\epsilon \to 0$,
  $$\|x^\epsilon - \bar x\|_{C^{\alpha'}} \to 0, \qquad
    \|A^\epsilon - A\|_{C^\beta} \to 0 \quad\text{and}\quad
    \|\tilde A^\epsilon - \tilde A\|_{C^\beta} \to 0$$
  where the convergence above is in probability.
\end{lemma}
\begin{proof}
  The first convergence is the fractional averaging theorem of \cite{Hairer-Li:20} (Theorem A
  therein, cf.\ also Theorem 3.13 there), which provides
  convergence in probability in $C^{\alpha'}$ for every $\alpha' < H$. For the second, since
  convergence in probability is equivalent to almost sure convergence along a further subsequence
  of every subsequence, it suffices to show that $\|A^\epsilon - A\|_{C^\beta} \to 0$ almost
  surely whenever $\|x^\epsilon - \bar x\|_{C^{\alpha'}} \to 0$ almost surely.

  Thus, assume the latter convergence holds almost surely and denote
  $x^{\theta, \epsilon}_t = \theta x^\epsilon_t + (1 - \theta)\bar x_t$ for $\theta \in [0, 1]$, so
  that $A^\epsilon_t = \int_0^1 D\bar f(x^{\theta, \epsilon}_t)\dd \theta$. For every
  $\theta \in [0, 1]$ we have $x^{\theta, \epsilon} \to \bar x$ in $C^{\alpha'}$, and so, by
  Lemma~\ref{lem:DFn-cvg},
  $\|D\bar f(x^{\theta, \epsilon}) - D\bar f(\bar x)\|_\beta \to 0$ as $\epsilon \to 0$.
  Moreover, by Equation~\eqref{eq:DF-bound}, we have the uniform bound
  \begin{align*}
    \sup_\epsilon \sup_{\theta \in [0, 1]} \|D\bar f(x^{\theta, \epsilon}) - D\bar f(\bar x)\|_\beta
     & \le \sup_\epsilon \sup_{\theta \in [0, 1]} \|D^2 \bar f\|_\infty (\|x^{\theta, \epsilon}\|_{\alpha'} + \|\bar x\|_{\alpha'}) \\
     & \le \|D^2 \bar f\|_\infty (\sup_\epsilon \|x^\epsilon\|_{\alpha'} + 2 \|\bar x\|_{\alpha'})<\infty,
  \end{align*}
  which follows from the convergence of $x^\epsilon \to \bar x$ in $C^{\alpha'}$.
  The dominated convergence theorem then yields
  $$\lim_{\epsilon \to 0}\|A^\epsilon - A\|_\beta
    \le \lim_{\epsilon \to 0} \int_0^1 \|D\bar f(x^{\theta, \epsilon}) - D\bar f(\bar x)\|_\beta \dd \theta = 0,$$
  as claimed. The same argument, applied with $\bar g$ in place of $\bar f$, yields the
  convergence of $\tilde A^\epsilon$ to $\tilde A$.
\end{proof}

Finally, we present the proof of Theorem~\ref{thm:main}.

\begin{proof}[of Theorem~\ref{thm:main}]
  Statement 1 is the content of Theorem~\ref{thm:VB-conv}. We prove Statement 2 with the exponent
  $\beta \in (1 - \alpha', \f 12)$ fixed above. Since $\alpha' \in (\f 12, H)$ may be chosen
  arbitrarily close to $H$, and since convergence in law in $C^\beta$ implies convergence in law
  in $C^{\alpha}$ for every $0 < \alpha \le \beta$, this yields the claim for every $\alpha \in (0, \f 12)$.

  We start by deducing the joint convergence of the inputs from the joint convergence in Statement 1.
  Let $F : C^{\alpha'}([0,T]; \R^m) \to C^{\alpha'}([0,T]; \R^d)$ denote the It\^o map of the
  effective equation \eqref{eq:bar-x}, that is, $F(b)$ is the unique
  solution of 
  $$x_t = x_0 + \int_0^t \bar f(x_s) \dd b_s + \int_0^t \bar g(x_s) \dd s.$$ 
  Writing this YDE in terms of the extended driver $(b, \iota) \in C^{\alpha'}([0,T]; \R^{m+1})$, where
  $\iota_t := t$, with the coefficient $(\bar f, \bar g) \in C^3_b(\R^d, \L(\R^{m+1}, \R^d))$, the
  map $F$ is well defined and continuous by Lemma~\ref{lem:young-ito-map} composed with the
  continuous affine map $b \mapsto (b, \iota)$. In particular $\bar x = F(B^H)$ almost
  surely. Moreover, by Lemma~\ref{lem:DFn-cvg}, the maps
  $\Lambda : x \mapsto D\bar f(x_\cdot)$, from $C^{\alpha'}([0,T]; \R^d)$ to
  $C^\beta([0,T]; \L(\R^d, \L(\R^m, \R^d)))$, and $\tilde{\Lambda} : x \mapsto D \bar g(x_\cdot)$,
  from $C^{\alpha'}([0,T]; \R^d)$ to $C^\beta([0,T]; \L(\R^d, \R^d))$, are
  continuous, and $A = \Lambda(F(B^H))$, $\tilde A = \tilde{\Lambda}(F(B^H))$ almost surely.

  By Theorem~\ref{thm:VB-conv}, applied with the exponents $\beta$ and $\alpha'$, and the
  continuous mapping theorem applied to the map
  $(v, b) \mapsto (v, b, F(b), \Lambda(F(b)), \tilde{\Lambda}(F(b)))$, we
  obtain
  $$\left(V^\epsilon, B^H, F(B^H), \Lambda(F(B^H)), \tilde{\Lambda}(F(B^H))\right)
    \stackrel{(d)}{\longrightarrow} \left(V, B^H, \bar x, A, \tilde A\right)$$
  in $C^\beta([0,T];C^3_b) \times C^{\alpha'} \times C^{\alpha'} \times C^\beta \times C^\beta$. 
  Thus, as $\|A^\epsilon - \Lambda(F(B^H))\|_{C^\beta} \to 0$ and
  $\|\tilde A^\epsilon - \tilde{\Lambda}(F(B^H))\|_{C^\beta} \to 0$ in probability by 
  Lemma~\ref{lem:A-eps-conv}, and $\|x^\epsilon - F(B^H)\|_{C^{\alpha'}} \to 0$, we have that
  \begin{equation}\label{eq:joint-conv-inputs}
    \left(V^\epsilon, B^H, x^\epsilon, A^\epsilon, \tilde A^\epsilon\right)
    \stackrel{(d)}{\longrightarrow} \left(V, B^H, \bar x, A, \tilde A\right).
  \end{equation}
  Recall that the map $(V, x) \mapsto \int_0^\cdot V(\dd s, x_s)$ is continuous from
  $C^\beta([0,T]; C^3_b) \times C^{\alpha'}([0,T]; \R^d)$ to $C^\beta([0,T]; \R^d)$. Together with 
  Lemma~\ref{lem:z-eps-NYDE} and Lemma~\ref{lem:G-eps-vanish},
  this yields that $\tilde{w}^\epsilon \to 0$ in probability in $C^\beta([0,T]; \R^d)$. Hence, applying the
  continuous mapping theorem to \eqref{eq:joint-conv-inputs} and Lemma~\ref{lem:z-eps-NYDE}, we obtain
  $$\left(w^\epsilon+\tilde{w}^\epsilon, A^\epsilon, \tilde A^\epsilon, B^H\right) \stackrel{(d)}{\longrightarrow}
    \left({\textstyle\int_0^\cdot V(\dd s, \bar x_s)}, A, \tilde A, B^H\right)
    \quad \text{in } C^\beta \times C^\beta \times C^\beta \times C^{\alpha'}.$$
  Finally, we are ready to conclude.
  By \eqref{eq:z-eps-as-Phi} and the continuity of the solution map $\Phi$ (noting that
  $(w, A, \tilde A, b) \mapsto \Phi(w, (A, \tilde A), (b, \iota))$ is a composition of $\Phi$ with
  continuous maps), a final application of
  the continuous mapping theorem yields
  $$z^\epsilon \stackrel{(d)}{\longrightarrow}
    \Phi\left({\textstyle\int_0^\cdot V(\dd s, \bar x_s)}, (D \bar f(\bar x), D \bar g(\bar x)), X^H\right) = z
    \quad\text{in } C^\beta([0,T]; \R^d),$$
  where $z$ is the unique solution of the limiting equation \eqref{eq:z-lim-eq}. Finally, we note
  that the law of $z$ is uniquely determined by the joint law of $(V, B^H)$,
  cf.\ Remark~\ref{rem:weak-uniqueness}. This completes the proof.
\end{proof}

\subsection*{Acknowledgements}
The authors would like to thank  the referees for their careful reading and  insightful comments, which 
resulted in a significant simplification of
part of the proof, via the use of nonlinear Young integration. 
  
Xue-Mei Li acknowledges support from the EPSRC grant EP/V026100/1 and the Swiss National Science 
Foundation project MINT (10000849). The authors also acknowledge support from NCCR SwissMAP.

\begin{appendix}

\section{Vanishing of the drift fluctuation field}\label{sec:drift-vanish}

In this subsection we show that the field associated to the fluctuations of the drift,
\begin{equation}\label{eq:G-eps-def}
  G^\epsilon_t(x) := \epsilon^{\f 12 - H} \int_0^t [g(x, y^\epsilon_r) - \bar g(x)] \dd r,
\end{equation}
already encountered in \eqref{eq:G-eps-def-intro}, vanishes as $\epsilon \to 0$. Note that, in
contrast to $V^\epsilon$, the field $G^\epsilon$ is defined pointwise in $(t, x)$ as an ordinary
Lebesgue integral. 

The quantitative statement is given in Lemma~\ref{lem:G-eps-vanish}. We shall first recall two
estimates from \cite{Hairer-Li:22}, adapted to the notation used here.
For $\zeta \in (0,1)$ and $\kappa > 0$, define the weighted space $\hat \CB$ of
$C^3$ functions $Z : \R^d \to \R^d$ by the norm
\begin{equ}
  \|Z\|_{\hat \CB} := \sup_{x \in \R^d} (1 + |x|)^{\kappa/2} \|Z\|_{\CC^{3 + \zeta}_x},
\end{equ}
where
\begin{equ}
  \|Z\|_{\CC^{3+\zeta}_x} = \sup_{|\ell|\le 3} |D^\ell Z(x)|
    + \sup_{0 < |x' - x|\le 1}\sup_{|\ell|\le 3}
      \f{|D^\ell Z(x)-D^\ell Z(x')|}{|x-x'|^\zeta}.
\end{equ}

\begin{lemma}[Lemma A.1 of \cite{Hairer-Li:22}]\label{lem:HL22-A1}
  The embedding $\hat \CB \subset C^3_b(\R^d,\R^d)$ is compact. Moreover, if
  $p \ge d$, $\zeta < 1-d/p$ and $\kappa > 4d/p$, then there exists a constant
  $C$ such that, for every $\hat \CB$-valued random function $Z : \R^d \to \R^d$,
  \begin{equ}\label{A1}
    \begin{aligned}
      \E \|Z\|_{\hat \CB}^p
      &\le C \sup_{x\in\R^d}\sup_{|\ell|\le 3}(1+|x|)^{\kappa p}\E |D^\ell Z(x)|^p\\
      &\quad + C \sup_{0 < |x-x'| \le 1}\sup_{|\ell| \le 3}
      (1+|x|)^{\kappa p} {\E |D^\ell Z(x)-D^\ell Z(x')|^p\over |x-x'|^p}.
    \end{aligned}
  \end{equ}
\end{lemma}

\begin{proposition}[Proposition 4.8 of \cite{Hairer-Li:22}]\label{prop:HL22-4-8}
  Assume Assumption~\ref{ass:ergodic} and let $k \ge 2$. There exist constants
  $c, C > 0$, depending only on $k$ and on the constants in Assumption~\ref{ass:ergodic},
  such that the following holds. Let $f_1,\ldots, f_k \in E_0$ satisfy
  $\int_\CY f_i \dd\mu = 0$, let $0 \le s_1 < \cdots < s_k$, and set
  $X_i = f_i(y_{s_i})$. Then
  \begin{equ}
    |\E_c X_{[k]}| \le C \exp\left(-c \sum_{i,j\le k} |s_i-s_j|\right)
    \prod_{i=1}^k \|f_i\|_{E_0},
  \end{equ}
  where $[k] = \{1, \dots, k\}$.
\end{proposition}

\begin{lemma}\label{lem:G-eps-vanish} 
  Let $ g$ satisfy Assumption~\ref{ass:C3} and $y_t$ satisfy Assumption~\ref{ass:ergodic}.
  Let $\beta \in (0, \f 12)$. Then, for all $p \ge 1$ and for all $\epsilon\in (0, 1]$, we have the bound
  \begin{equation}\label{eq:G-eps-vanish}
  \E \|G^\epsilon\|^p_{C^\beta([0,T]; C^3_b)} \lesssim \epsilon^{p(1 - H)}.
  \end{equation}
\end{lemma}
\begin{proof}
  Let $h \in E_0$ with $\int_\CY h \dd \mu = 0$ and set
  $$S^\epsilon_{s,t}(h) := \epsilon^{-\f 12} \int_s^t h(y^\epsilon_r) \dd r.$$
  We claim that, for every even integer $p \ge 2$,
  \begin{equation}\label{eq:S-bound}
    \|S^\epsilon_{s,t}(h)\|_{L^p} \lesssim  |t - s|^{\f 12} \|h\|_{E_0}.
  \end{equation}
  Consider first the case $|t - s| \le \epsilon$. By Minkowski's integral inequality and the
  stationarity of $y$,
  $$\Big\|\int_s^t h(y^\epsilon_r) \dd r\Big\|_{L^p} \le \int_s^t \|h(y^\epsilon_r)\|_{L^p} \dd r
    = (t-s) \|h\|_{L^p(\mu)} \lesssim (t-s) \|h\|_{E_0},$$
  using $E_0 \subset \bigcap_{q \ge 1} L^q(\CY, \mu)$. Hence
  $\|S^\epsilon_{s,t}(h)\|_{L^p} \lesssim \epsilon^{-\f 12} |t-s|\|h\|_{E_0}
  \le |t-s|^{\f 12}\|h\|_{E_0}$, where we used $|t - s|^{\f 12} \le \epsilon^{\f 12}$.

  Consider now the case $|t - s| \ge \epsilon$ and set $L := \epsilon^{-1}|t - s| \ge 1$. After the time
  rescaling $r \mapsto \epsilon r$,
  $$S^\epsilon_{s,t}(h) = \epsilon^{\f 12} \int_{s/\epsilon}^{t/\epsilon} h(y_r) \dd r.$$
  By Proposition~\ref{prop:HL22-4-8}, the joint cumulants of $h$ along the trajectory of $y$
  satisfy
  $$\left|\mathbb{E}_c(h(y_{r_1}), \dots, h(y_{r_k}))\right| \lesssim e^{-c\sum_{i < j} |r_i - r_j|} \|h\|^k_{E_0},$$
  so that, ordering the variables and integrating the relative coordinates, the $k$-th cumulant of
  $\int_{s/\epsilon}^{t/\epsilon} h(y_r)\dd r$ is bounded by $C_k L \|h\|^k_{E_0}$ for every
  $k \ge 2$, while the first cumulant vanishes since $h$ is mean-zero and $y$ is stationary.
  Expressing moments in terms of cumulants by \eqref{cul-relation2}, every contributing partition
  consists of blocks of size at least two, hence has at most $p/2$ blocks, and since $L \ge 1$,
  $$\E\Big|\int_{s/\epsilon}^{t/\epsilon} h(y_r) \dd r\Big|^p \lesssim L^{p/2} \|h\|^p_{E_0}.$$
  Consequently
  $\|S^\epsilon_{s,t}(h)\|_{L^p} \lesssim \epsilon^{\f 12} L^{\f 12} \|h\|_{E_0}
  = |t - s|^{\f 12} \|h\|_{E_0}$, proving \eqref{eq:S-bound}.
  Now, noting that
  
  $G^\epsilon_{s,t}(x) = \epsilon^{1 - H} S^\epsilon_{s,t}(g(x, \cdot) - \bar g(x))$ and, more generally, for any
  multi-index $\ell$ with $|\ell| \le 3$ and $x, x' \in \R^d$ such that $|x - x'| \le 1$,
  \begin{equs}
    D^\ell_x G^\epsilon_{s,t}(x) &= \epsilon^{1 - H} S^\epsilon_{s,t}\big(D^\ell_x [g - \bar g](x, \cdot)\big),\\
    D^\ell_x G^\epsilon_{s,t}(x) - D^\ell_x G^\epsilon_{s,t}(x')
    & = \epsilon^{1 - H} S^\epsilon_{s,t}\big(D^\ell_x [g - \bar g](x, \cdot) - D^\ell_x [g - \bar g](x', \cdot)\big),
  \end{equs}
  where, by Assumption~\ref{ass:C3},
  \begin{equs}
    \|D^\ell_x [g - \bar g](x, \cdot)\|_{E_0} & \lesssim (1 + |x|)^{-\kappa}, \\
    \|D^\ell_x [g - \bar g](x, \cdot) - D^\ell_x [g - \bar g](x', \cdot)\|_{E_0} & \lesssim (1 + |x|)^{-\kappa}|x - x'|,
  \end{equs}
  the latter using one further derivative of $g$. Combining this with Equation~\eqref{eq:S-bound}, we obtain
  \begin{equs}
    \|D^\ell_x G^\epsilon_{s,t}(x)\|_{L^p} 
    & \lesssim \epsilon^{1 - H}(1 + |x|)^{-\kappa}|t - s|^{\f 12}\\
    \|D^\ell_x G^\epsilon_{s,t}(x) - D^\ell_x G^\epsilon_{s,t}(x')\|_{L^p} 
    & \lesssim \epsilon^{1 - H}(1 + |x|)^{-\kappa}|x - x'||t - s|^{\f 12}.
  \end{equs}
  Thus, by Lemma~\ref{lem:HL22-A1} applied to the preceding pointwise and increment estimates, and
  an application of the Kolmogorov continuity criterion,
  we have the bound
  $$\E \|G^\epsilon\|^p_{C^\beta([0,T]; C^3_b)} \lesssim \epsilon^{p(1 - H)}$$
  for even $p$ satisfying $\beta < \f 12 - 1/p$. Consequently, 
  $\|\|G^\epsilon\|_{C^\beta([0,T]; C^3_b)}\|_p \lesssim \epsilon^{1 - H}$ for all $p$ as we can bound 
  the lower $L^p(\Omega)$-norms by the higher ones. 
\end{proof}

\section{Proof of Lemma~\ref{lem:GD-kappa}}\label{sec:proof-GD-kappa}
In this section we prove Lemma~\ref{lem:GD-kappa}. The proof is based on the following auxiliary
result, which allows us to estimate the integral $\tilde \CJ^\epsilon(\Delta,p)$
by reducing to elementary base cases. Let $B$ be a finite set. Recall that  $\mathcal{G}^S(B)$
consists of $(\Delta, p)$ where $\Delta$ is a partition of $B$ and $p$ is a partition of $B$
consisting of pairings, and $(\Delta, p)$ is connected if $\Delta \vee p = \{B\}$, cf. Section~\ref{sec:cumulant}.

\begin{lemma}\label{lem:restrict-G}
  Let $B$ be a finite index set and $B'\subset B$. Let $(\Delta,p)\in \mathcal{G}^S(B)$ such that
  $\Delta$ has no singletons. Denote the restrictions of  $\Delta$ and $p$ to $B'$ by $\Delta'$ and 
  $p'$ respectively. Assume that $\Delta'$ has no singletons, that $p'$ is made of pairs, and that 
  $(\Delta', p')$ is connected. Let us also recall the notation
  $$\tilde \CJ^\epsilon(\Delta, p) = \epsilon^{ \frac{|B|}{2}}\int_{[0, \frac{T}{\epsilon}]^B}
    \prod_{\substack{A \in \Delta}} e^{-c \sum_{i,j \in A}|r_i-r_j|}
    \prod_{\substack{\{a,b\} \in p }}|r_a - r_b|^{2H-2} \dd r^B.$$
  Then the following holds
  $$ \tilde \CJ^\epsilon(\Delta, p) \lesssim  \tilde \CJ^\epsilon(\Delta', p').$$
\end{lemma}
\begin{proof}
  In the following, if $a,b$ belong to the same set in a partition $\Delta$, we write 
  $a\sim_{\Delta}b$. We also write $[a]_\Delta$ for the equivalence class of $a$ under the relation
  $\sim_\Delta$. The assumptions that $p'$ is made of pairs mean that if a pair $\pi$ belongs to
  $p$ then either $\pi \subseteq B'$ or $\pi \cap B' = \varnothing$.

  Using Fubini and the above remark, we can write that
  $$\tilde \CJ^\epsilon(\Delta, p) = \epsilon^{ \frac{|B|}{2}}\int_{[0, \frac{T}{\epsilon}]^{B'}}
    \prod_{\substack{A \in \Delta'}} e^{-c \sum_{i,j \in A}|r_i-r_j|}
    \prod_{\substack{\{a,b\} \in p' }}|r_a - r_b|^{2H-2} F(r^{B'}) \dd r^{B'},$$
  where we introduced the notation
  $$F(r^{B'}) := \int_{[0, \frac{T}{\epsilon}]^{B\setminus B'}}
    \prod_{\substack{A \in \Delta}} e^{- c \sum_{i,j \in A, i \in B\setminus B'}|r_i-r_j|}
    \prod_{\substack{\{a,b\} \in p\setminus p' }}|r_a - r_b|^{2H-2} \dd r^{B\setminus B'}\:.$$
  From this, we have by a supremum bound that $\tilde \CJ^\epsilon(\Delta,p) \leq 
  \epsilon^{\frac{|B\setminus B'|}{2}} \tilde \CJ^\epsilon(\Delta', p')  \sup_{r^{B'} \in \mathbb{R}^{B'}} F(r^{B'})$.
  Consequently, the claim holds if one has
  $$(\star) := \sup_{r^{B'} \in \mathbb{R}^{B'}} F(r^{B'}) \lesssim \epsilon^{-\frac{|B\setminus B'|}{2}}.$$
  To see why this holds, first assume that we are given a bi-partition $C^+, C^-$ of the set $B\setminus B'$ such that
  the following hold:
  \begin{itemize}
    \item [(1)] Each pair $\pi \in p\setminus p'$ is of the form $\pi = \{a^+, a^-\}$ with $a^+ \in C^+$ and $a^- \in C^-$.
    \item [(2)] For any $ a^+ \in C^+$,
          there exists some $\phi(a^+) \in B'\cup C^-$ such that $a^+ \sim_\Delta \phi(a^+)$.
  \end{itemize}
  Then, upper bounding by one all terms in the product that are not of the form $e^{-c|r_{a^+} - r_{\phi(a^+)}|}$
  yields
  $$(\star) \leq  \sup_{r^{B'} \in \mathbb{R}^{B'}} \int_{[0, \frac{T}{\epsilon}]^{C^-}}
    \Big(\int_{[0, \frac{T}{\epsilon}]^{C^+}} \prod_{a^+\in C^+}e^{-c|r_{a^+} 
      - r_{\phi(a^+)}|}|r_{a^+}-r_{a^-}|^{2H-2} \dd r^{C^+}\Big)\dd r^{C^-}.$$
  Thanks to the assumptions on the bi-partition constructed, the inner integral can be factorized
  into single-variable integrals from which we obtain the estimate
  \begin{align*}
        & \int_{[0, \frac{T}{\epsilon}]^{C^+}} 
          \prod_{a\in C^+}e^{-c|r_a - r_{\phi(a)}|}|r_a-r_{a^-}|^{2H-2} \dd r^{C^+}        \\
    =\  & \prod_{a^+\in C^+} \int_{[0, \frac{T}{\epsilon}]}e^{-c|r_{a^+} - r_{\phi(a^+)}|}
      |r_{a^+}-r_{a^-}|^{2H-2} \dd r_{a^+}
      \leq \left(\int_\mathbb{R} e^{-c|x|}|x|^{2H-2}\dd x\right)^{|C^+|}
  \end{align*}
  in which the inequality follows from the Hardy-Littlewood rearrangement inequality
  (cf. \cite[Theorem 3.4]{Lieb-Loss:01}). Namely, the convolution of the two symmetric decreasing
  kernels is maximised when their centres coincide. Moreover, the right-hand side above is finite 
  as $2H-2 > -1$. Consequently, 
  $(\star) \lesssim \epsilon^{-|C^{-}|}$ holds, which is precisely the desired bound since
  $|C^-| = |B \setminus B'|/2$.

  We are left with arguing that one can construct a bi-partition with the above properties.
  This is done by induction. We begin with two empty sets $C_0^+ = C_0^- = \varnothing$ and
  construct $C_{i + 1}^{\pm}$ as follows. Assuming $C_i^{\pm}$ are constructed, pick any new pair $\{a,b\}
    \in p\setminus p'$ such that $[a]_\Delta\cap (B'\cup C_i^{-}) \neq \varnothing$,
  and let $C_{i+1}^+ = C_i^+\cup\{a\},C_{i+1}^- = C_i^-\cup\{b\}$.
  Such a pair exists whenever not all pairs have been assigned: by connectedness, some $\Delta$-block
  must connect the assigned set to an unassigned pair. If this block meets some already assigned
  $C_i^+$ element, then by the induction hypothesis it also meets $B'\cup C_i^-$, so the same
  choice criterion still applies.
  This step can be repeated until all pairs in $p\setminus p'$ have been assigned to $C^+$ or $C^-$, 
  and produces a bi-partition as described above. This concludes the proof.
\end{proof}

From this auxiliary result we can now derive Lemma~\ref{lem:GD-kappa}. Let us recall the shorthand
notation $[N] = \{1,\cdots, N\}$ for any integer $N\geq 1$.
\begin{corollary}\label{cor:GD-kappa}
  Let  $(\Delta, p) \in \mathcal{G}^S(B)$ be such that $|S_\Delta| = 0$, and $|B|>2$.
  Then, one has
  $$\tilde \CJ^\epsilon(\Delta,p) \lesssim \epsilon^{2-2H}.$$
\end{corollary}
\begin{proof}
  We distinguish two cases, whether $(\Delta,p)$ contains a \textit{cycle} or not. In either case,
  we use the Lemma above to construct a suitable $(\Delta',p')$ for which the bound
  can be proved directly.

  If $(\Delta,p)$ contains a cycle, (i.e. there exists $\{v_i, u_i\}_{i = 1}^N \subseteq B$ for
  some $N \ge 2$, such that setting $u_{N + 1} = u_1$
  we have $u_i \sim_\Delta v_i$ and $v_i \sim_p u_{i+1}$ for all $i \in [N]$),
  let $B'$ denote the set of all elements of the cycle.  By construction, $B'$ satisfies the assumptions
  of Lemma~\ref{lem:restrict-G}, namely, it is connected and using the same notations, $\Delta'$ 
  has no singletons and $p'$ is made of pairs.

  Identifying the integration variables with elements of the index set for simplicity 
  ($r_{u_i} = u_i, r_{v_i} = v_i$), one has
  $$\tilde \CJ^\epsilon(\Delta',p') \lesssim \epsilon^N\int_{[0, \frac{T}{\epsilon}]^{2N}}
    \prod_{i=1}^Ne^{-c|u_i-v_i|}|v_i-u_{i+1}|^{2H-2} \dd u^{[N]} \dd v^{[N]}.$$

  This integral can now be estimated directly. Setting $p_i = u_i -v_i, q_i = v_i-u_{i+1}$, 
  notice that $\sum_{i=1}^N p_i + \sum_{i=1}^N q_i = 0$. Consequently, there exists some $q_j$ such that $|q_j|
    \geq\frac{1}{N}|\sum_{i=1}^N p_i |$. Without loss of generality, assume by symmetry that $j = N$.
  Since $2H-2<0$, this gives $|q_N|^{2H-2}\lesssim |\sum_{i=1}^N p_i|^{2H-2}$, and therefore
  $$\tilde \CJ^\epsilon(\Delta',p') \lesssim  \epsilon^N\int_{[0, \frac{T}{\epsilon}]^{2N}}
    e^{-c\sum_{i=1}^N| p_i |}\left|\sum_{i=1}^N p_i \right|^{2H-2}\prod_{i=1}^{N-1}|q_i|^{2H-2}  \dd u^{[N]} \dd v^{[N]}.$$
  Performing the change of variables $u,v \mapsto (p_i)_{i = 1,\cdots,N},(q_i)_{i = 1,\cdots,N-1},u_1$ 
  (which has a Jacobian equal to one), $u_1$ can now be integrated out. The variables $p_i,q_i$
  are differences of variables in $[0,T/\epsilon]$, so their range is contained in
  $[-K/\epsilon,K/\epsilon]$ for some $K \geq T$, and the domain of integration can be extended to get
  $$\tilde \CJ^\epsilon(\Delta',p') \lesssim  \epsilon^{N-1}\int_{[\frac{-K}{\epsilon},\frac{K}{\epsilon}]^{2N-1}}
    e^{-c\sum_{i=1}^N| p_i |}\left|\sum_{i=1}^N p_i \right|^{2H-2}\prod_{i=1}^{N-1}|q_i|^{2H-2}  \dd p^{[N]} \dd q^{[N-1]}.$$
  The upper bounding integral now factorizes as
  $$\tilde \CJ^\epsilon(\Delta',p') \lesssim \epsilon^{N-1}\left( \int_{[-\frac{K}{\epsilon},\frac{K}{\epsilon}]}
    |q|^{2H-2} \dd q\right)^{N-1}\int_{\mathbb{R}^N} e^{-c\sum_{i=1}^N| p_i |} \left|\sum_{i=1}^N p_i \right|^{2H-2} \dd p^{[N]} $$
  where the rightmost integral is finite, yielding $ \tilde \CJ^\epsilon(\Delta',p') \lesssim \epsilon^{N-1-(2H-1)(N-1)}
    \leq \epsilon^{2-2H}$.

  For the second case, when $(\Delta,p)$ contains no cycle, the graph obtained by contracting the
  blocks of $\Delta$ and retaining the non-internal pairs of $p$ is a tree. Thus there exist two disjoint sets
  $A, A'\in \Delta$, each of them connected to $B\setminus(A\cup A')$ by a single pair $\pi_A = 
  \{v_1,u_2\}, \pi_{A'} = \{v_N,u_{N+1}\} \in p$ respectively. Moreover, as $|S_\Delta| = 0$, each
  of $A$ and $A'$ contains an internal pair of $p$; otherwise the unique element paired outside would
  be the only element of the corresponding block. Let us denote those pairs by
  $\{x_1, x_2\}$ and $\{y_1, y_2\}$ respectively.
  By connectedness, one can pick a path
  $$A \ni v_1 \sim_p u_2 \sim_\Delta v_2 \sim_p \cdots \sim_\Delta v_N \sim_p u_{N + 1} \in A'$$
  from $A$ to $A'$. Denoting
  $$B' = \{v_i : i \in [N]\} \cup \{u_i : 2 \le i \le N+1\} \cup \{x_1, x_2, y_1, y_2\},$$
  it is clear that $B'$ falls under the assumptions of Lemma~\ref{lem:restrict-G}, and defines a 
  $(\Delta',p')$ for which $\tilde \CJ^\epsilon(\Delta',p')$ is given by
  \begin{align*}
    \epsilon^{N + 2} \int_{[0, \frac{T}{\epsilon}]^{2N+4}} & e^{-c(|x_1-x_2|+|x_1-v_1|+|x_2-v_1|)}
    |x_1-x_2|^{2H-2}\prod_{i=2}^Ne^{-c|u_i-v_i|}                                 \\
                                                                            & \prod_{i=1}^N|v_i-u_{i+1}|^{2H-2}
    e^{-c(|y_1-y_2|+|y_1-u_{N+1}|+|y_2-u_{N+1}|)} |y_1-y_2|^{2H-2}                    \\
                                                                            & \dd u_2 \cdots \dd u_{N+1}\dd v^{[N]}\dd x_1\dd x_2 \dd y_1 \dd y_2
  \end{align*}
  where $N + 2 = \frac{|B'|}{2}$.

  Now, integrating with respect to $\dd x_1 \dd x_2 \dd y_1 \dd y_2$ yields a finite contribution, 
  leaving us with
  $$\epsilon^{-N-2} \tilde \CJ^\epsilon(\Delta',p') \lesssim\int_{[0, \frac{T}{\epsilon}]^{2N}}
    \prod_{i=2}^Ne^{-c|u_i-v_i|}\prod_{i=1}^N |v_i-u_{i+1}|^{2H-2}
    \dd u_2\cdots \dd u_{N+1}\dd v^{[N]}.$$
  Writing $L=T/\epsilon$, integrate successively in $u_2,\ldots,u_N$ against the adjacent exponential
  factors. Each such integration is bounded uniformly, and the remaining variables contribute at most
  $L^{N-1}$ times one fractional double integral. Hence
  $$\tilde \CJ^\epsilon(\Delta',p') \lesssim \epsilon^{N+2}L^{N-1}\int_0^L\int_0^L|v-u|^{2H-2}\dd u\dd v
    \lesssim \epsilon^{3-2H}\leq \epsilon^{2-2H}.$$
  This concludes the proof in this case.
\end{proof}

\section{The rough path lift from the previous approach}\label{sec:rough-lift-appendix}

This appendix records the rough path construction used in version one of the current paper \cite{arxiv}, 
no longer needed for the proof of the main theorem here after the change of the proof strategy. 
Nevertheless, the convergence of the rough path associated with the key path  $V^\epsilon$ offers 
value for comparison and further studies. Here we illustrate this result and explain how it could be used.
 
\subsection{Construction of the lifts}\label{sec:construction-rp}
Let us recall that
\begin{equ}
  V_{s,t}^\epsilon(x)= \epsilon^{\f 12-H}\int_s^t \left(f(x, y^\epsilon_r)
  - \bar f(x)\right) \dd B^H_r,\qquad x\in \mathbb{R}^d,
\end{equ}
and
\begin{equ}\label{eq:appendix-U-eps}
  U^\epsilon_t(x) = \int_0^t f(x, y^\epsilon_r) \dd B^H_r,\qquad
  U_t(x)= \bar f(x)B_t^H .
\end{equ}
Thus
\begin{equ}\label{eq:appendix-U-eps-decomp}
  U^\epsilon_{s,t}=U_{s,t}+\epsilon^{H-\f 12}V^\epsilon_{s,t}.
\end{equ}
As in \cite{Hairer-Li:22}, the second order process $\mathbb{V}^\epsilon$ is constructed by
regularising the fractional Brownian motion. If $\phi$ is a Schwartz function with integral one,
$\phi_\delta(t)=\delta^{-1}\phi(t/\delta)$, and $B^\delta=B^H*\phi_\delta$, set
$$V^{\epsilon,\delta}_t(x)=\epsilon^{\f 12-H}
  \int_0^t \left(f(x, y^\epsilon_r) - \bar f(x)\right)\dot B^\delta_r \dd r,
  \qquad
  \mathbb{V}^{\epsilon,\delta}_{s,t}(x,x')
  = \int_s^t V^{\epsilon,\delta}_{s,r}(x)\dd V^{\epsilon,\delta}_r(x').$$
The following combines the parts of \cite[Propositions~3.1 and~5.1]{Hairer-Li:22} used here.

\begin{proposition}\label{prop:def-Veps}
  Let $\CB=C_b^3(\R^d,\R^d)$.
  Let $f$ satisfy Assumption~\ref{ass:C3}, and let $\alpha\in(\f13,\f12)$. Then, for every
  $\epsilon>0$, $(V^{\epsilon,\delta},\mathbb{V}^{\epsilon,\delta})$ converges in probability as
  $\delta\to0$ to an $\alpha$-H\"older geometric rough path
  $\mathbf V^\epsilon=(V^\epsilon,\mathbb V^\epsilon)$. 
  Moreover $\mathbf V^\epsilon$ converges in law to a Gaussian
  rough path $\mathbf V=(V,\mathbb V)$ in this topology. The path component $V$ has covariance
  $$\mathbb{E}[V_{s,t}(x)\otimes V_{u,v}(x')]
    = |[s,t]\cap[u,v]|\left(\Sigma(x,x')+\Sigma(x',x)^\top\right),$$
  where $\Sigma$ is defined in \eqref{e:defSigma}, and
  $$\mathbb V_{s,t}(x,x') = \int_s^t V_{s,r}(x)\otimes \dd V_r(x')+(t-s)\Sigma(x,x'),$$
  where the integral is understood in the It\^o sense.
\end{proposition}

Fix $\alpha\in(\f13,\f12)$ and $\alpha'\in(\f12,H)$ with $\alpha+\alpha'>1$.
Since $U$ is $\alpha'$-H\"older, the mixed integrals between $V^\epsilon$ and $U$ are Young
integrals. We define
\begin{equation}\label{eq:lift-u-eps}
  \mathbb{U}^\epsilon_{s,t} := \mathbb{U}_{s,t}
  + \epsilon^{H-\frac{1}{2}}
  \left( \int_s^t V^\epsilon_{s,r}\otimes \dd U_r
  + \int_s^t U_{s,r}\otimes \dd V^\epsilon_r \right)
  + \epsilon^{2H-1}\mathbb{V}^\epsilon_{s,t},
\end{equation}
where
\[
  \mathbb{U}_{s,t}=\int_s^t U_{s,r}\otimes \dd U_r .
\]
This is the lift suggested by the decomposition \eqref{eq:appendix-U-eps-decomp}.
Writing $W^\epsilon=(V^\epsilon,U^\epsilon)$, set
\begin{equation}\label{eq:W-epsilon}
  \mathbb{W}^\epsilon_{s,t}:=
  \begin{pmatrix}
    \mathbb{V}^\epsilon_{s,t} & \int_s^t V^\epsilon_{s,r}\otimes \dd U^\epsilon_r \\
    \int_s^t U^\epsilon_{s,r}\otimes \dd V^\epsilon_r & \mathbb{U}^\epsilon_{s,t}
  \end{pmatrix},
\end{equation}
where
\begin{align*}
  \int_s^t V^\epsilon_{s,r}\otimes \dd U^\epsilon_r
  & := \int_s^t V^\epsilon_{s,r}\otimes \dd U_r
    + \epsilon^{H-\frac{1}{2}}\mathbb{V}^\epsilon_{s,t},\\
  \int_s^t U^\epsilon_{s,r}\otimes \dd V^\epsilon_r
  & := \int_s^t U_{s,r}\otimes \dd V^\epsilon_r
    + \epsilon^{H-\frac{1}{2}}\mathbb{V}^\epsilon_{s,t}.
\end{align*}
Similarly, for $\bar W^\epsilon=(V^\epsilon,U)$, define
\begin{equation}\label{eq:W-eps-bar-def}
  \bar{\mathbb{W}}^\epsilon_{s,t}:=
  \begin{pmatrix}
    \mathbb{V}^\epsilon_{s,t} & \int_s^t V^\epsilon_{s,r}\otimes \dd U_r \\
    \int_s^t U_{s,r}\otimes \dd V^\epsilon_r & \mathbb{U}_{s,t}
  \end{pmatrix},
\end{equation}
and write $\bar{\mathbf W}^\epsilon=(\bar W^\epsilon,\bar{\mathbb W}^\epsilon)$.

\begin{lemma}\label{lem:appendix-W-lift}
  The pairs $\bar{\mathbf W}^\epsilon$ and $\mathbf W^\epsilon=(W^\epsilon,\mathbb W^\epsilon)$
  are geometric rough paths.
\end{lemma}

\begin{proof}
  The Young identity
  $$\delta\left(\int_s^t X_{s,r}\otimes \dd Y_r\right)_{s,u,t} = X_{s,u}\otimes Y_{u,t}$$
  holds whenever $X\in C^\alpha$, $Y\in C^\beta$, and $\alpha+\beta>1$. Together with Chen's relation
  for $\mathbb V^\epsilon$ and $\mathbb U$, this gives Chen's relation for
  $\bar{\mathbb W}^\epsilon$. The weakly geometric relation follows from the geometricity of
  $\mathbf V^\epsilon$ and the corresponding Young identity for the symmetric part of the mixed
  integrals.

  Finally, $\mathbf W^\epsilon$ is the image of $\bar{\mathbf W}^\epsilon$ under the linear map
  $(a,b)\mapsto(a,b+\epsilon^{H-\f12}a)$. Hence $\mathbf W^\epsilon$ is geometric as well, and its
  second level is exactly \eqref{eq:W-epsilon}.
\end{proof}

\subsection{Continuity of the joint lift}
The deterministic input used in the earlier approach is the continuity of the operation which couples a
rough path with a sufficiently regular Young path.

\begin{lemma}\label{lem:young-joint-lift-cont}
  Let $\alpha\in(\f13,\f12)$ and $\rho\in(\f12,1)$ satisfy $\alpha+\rho>1$.
  Let $\mathbf X^n=(X^n,\mathbb X^n)$ and $\mathbf X=(X,\mathbb X)$ be weakly geometric
  $\alpha$-H\"older rough paths over a Banach space $E$, and let $Y^n,Y$ be $\rho$-H\"older paths
  with values in a Banach space $F$. If $\mathbf X^n\to\mathbf X$ in
  the $\alpha$-H\"older rough path topology and $Y^n\to Y$ in $C^\rho([0,T],F)$, then the joint lifts
  $$\mathbb A^n_{s,t}:=
    \begin{pmatrix}
      \mathbb X^n_{s,t} & \int_s^t X^n_{s,r}\otimes \dd Y^n_r \\
      \int_s^t Y^n_{s,r}\otimes \dd X^n_r & \int_s^t Y^n_{s,r}\otimes \dd Y^n_r
    \end{pmatrix}$$
  converge to the corresponding joint lift $\mathbb A$ in the rough path topology.
  Consequently, convergence in law of $(\mathbf X^n,Y^n)$ implies convergence in law of the joint
  lifts.
\end{lemma}

\begin{proof}
  This is the standard continuity of the Young integral in H\"older norms. The top left entry
  converges by assumption. The other entries converge because the maps
  $(X,Y)\mapsto\int X_{s,r}\otimes \dd Y_r$ and
  $(Y,X)\mapsto\int Y_{s,r}\otimes \dd X_r$ are locally Lipschitz from
  $C^\alpha\times C^\rho$ to $C^{\alpha+\rho}$ when $\alpha+\rho>1$, while
  $Y\mapsto\int Y_{s,r}\otimes \dd Y_r$ is locally Lipschitz on $C^\rho$ when $\rho>\f12$.
  The assertion for convergence in law follows from the continuous mapping theorem.
\end{proof}

\subsection{Outline of the previous convergence argument}
The previous proof used the rough path $\mathbf W^\epsilon=(W^\epsilon,\mathbb W^\epsilon)$ as the
driver. In the notation above, the desired convergence was
\begin{equation}\label{eq:appendix-old-target}
  \mathbf W^\epsilon \Rightarrow \mathbf W
\end{equation}
in the $\alpha$-H\"older rough path topology over $\CB\oplus\CB$, where
$W=(V,U)$, $U_t(x)=\bar f(x)B^H_t$, and
$$\mathbb W_{s,t}=
  \begin{pmatrix}
    \mathbb V_{s,t} & \int_s^t V_{s,r}\otimes \dd U_r \\
    \int_s^t U_{s,r}\otimes \dd V_r & \mathbb U_{s,t}
  \end{pmatrix}.$$
The mixed integrals in the limit are Young integrals, while $\mathbb V$ is the It\^o lift described in
Proposition~\ref{prop:def-Veps}. Once \eqref{eq:appendix-old-target} is known, the old proof applies
the continuity of the rough differential equation driven by $\mathbf W^\epsilon$.

The convergence \eqref{eq:appendix-old-target} was not meant to follow only from the marginal
convergence of $\mathbf V^\epsilon$ and the first level joint convergence of $(V^\epsilon,B^H)$.
One has to prove an enhanced joint convergence statement. A convenient way to organize the proof is to
first work with $\bar{\mathbf W}^\epsilon=(\bar W^\epsilon,\bar{\mathbb W}^\epsilon)$, where
$\bar W^\epsilon=(V^\epsilon,U)$, and then compare it to $\mathbf W^\epsilon$ using
\eqref{eq:appendix-U-eps-decomp}.

At the first level, one considers finite collections of increments and linear combinations of the form
$$W^\epsilon_a
  = \lambda_a \bar f_{i_a}(x_a)B^H_{s_a,t_a}
  + \eta_a\epsilon^{\f12-H}\int_{s_a}^{t_a}
    \left(f_{i_a}(x_a,y^\epsilon_r)-\bar f_{i_a}(x_a)\right)\dd B^H_r.$$
Conditioning on $y$, applying Wick's formula to $B^H$, and then expanding moments into cumulants gives
the same connected partition expansion as in Section~\ref{sec:cumulant}. The only difference is that
the deterministic $U$-part produces singleton blocks. After the rescaling $r\mapsto r/\epsilon$, each
singleton carries a factor $\epsilon^{H-\f12}\bar f_{i_a}(x_a)$. The estimates used in the proof of
Lemma~\ref{lem:cumulant-conv}, together with Proposition~\ref{prop:HL22-4-8}, then show that all
cumulants of order strictly larger than two vanish. The second cumulants converge to the covariance of
$(V,U)$. The mixed covariance between the two components is zero, since $y$ is independent of $B^H$ and
the centered observable has mean zero. Hence the limiting first level is Gaussian and the two components
$V$ and $U$ are independent.

The remaining point is the second level. Proposition~\ref{prop:def-Veps} gives the marginal convergence
of $\mathbf V^\epsilon=(V^\epsilon,\mathbb V^\epsilon)$. For the rough path proof one needs this
convergence jointly with $U$, namely
$$(\mathbf V^\epsilon,U)\Rightarrow(\mathbf V,U).$$
This is the enhanced joint convergence which must be verified separately. To reproduce the earlier
argument, one works first with the smoothed lift
$(V^{\epsilon,\delta},\mathbb V^{\epsilon,\delta})$ and includes variables of the form
$$\int_s^t V^{\epsilon,\delta}_{s,r}(x)\dd V^{\epsilon,\delta}_r(x')$$
together with the $U$-increments in the same cumulant expansion. Wick's formula again produces
connected pairings. The estimates used for Lemma~\ref{lem:cumulant-conv} control all higher cumulants,
while the second cumulants give the It\^o correction in Proposition~\ref{prop:def-Veps} and no mixed
covariance with $U$. After tightness, one sends $\delta\to0$ and obtains the displayed enhanced
convergence. With this input, Lemma~\ref{lem:young-joint-lift-cont} gives the convergence of the mixed
Young terms in \(\bar{\mathbb W}^\epsilon\), and therefore
$$\bar{\mathbf W}^\epsilon\Rightarrow\mathbf W.$$
Finally, the difference between $\mathbf W^\epsilon$ and $\bar{\mathbf W}^\epsilon$ is small. The
first levels differ by $(0,\epsilon^{H-\f12}V^\epsilon)$, and the second levels differ only through
terms of the form
$$\epsilon^{H-\f12}\mathbb V^\epsilon,\qquad
  \epsilon^{H-\f12}\int_s^t V^\epsilon_{s,r}\otimes \dd U_r,\qquad
  \epsilon^{H-\f12}\int_s^t U_{s,r}\otimes \dd V^\epsilon_r,\qquad
  \epsilon^{2H-1}\mathbb V^\epsilon.$$
These terms vanish in probability by the tightness of $\mathbf V^\epsilon$ and Young's inequality.
Thus $\mathbf W^\epsilon$ and $\bar{\mathbf W}^\epsilon$ have the same weak limit.

This appendix is only meant to record the structure of the earlier rough path argument. The main body
of the paper proves the first level joint convergence of $(V^\epsilon,B^H)$ in
Theorem~\ref{thm:VB-conv}, which is sufficient for the nonlinear Young argument and the residue lemma
used in the proof of Theorem~\ref{thm:main}.

\end{appendix}

\bibliographystyle{alpha}
\bibliography{bibliography}

\end{document}

%% file: prelude.tex
\mathtoolsset{showonlyrefs=true}

\newif\ifdark
\IfFileExists{dark}{\darktrue}{\darkfalse}

\ifdark

\definecolor{darkred}{rgb}{0.9,0.2,0.2}
\definecolor{darkblue}{rgb}{0.7,0.3,1}
\definecolor{darkgreen}{rgb}{0.1,0.9,0.1}
\definecolor{pagebackground}{rgb}{.15,.21,.18}
\definecolor{pageforeground}{rgb}{.84,.84,.85}
\pagecolor{pagebackground}
\AtBeginDocument{\globalcolor{pageforeground}}

\else

\definecolor{darkred}{rgb}{0.7,0.1,0.1}
\definecolor{darkblue}{rgb}{0.4,0.1,0.8}
\definecolor{darkgreen}{rgb}{0.1,0.7,0.1}
\definecolor{pagebackground}{rgb}{1,1,1}
\definecolor{pageforeground}{rgb}{0,0,0}

\fi

\definecolor{newred}{RGB}{208,16,76}
\definecolor{newgreen}{RGB}{34,125,81}

\makeatletter
\newcommand{\globalcolor}[1]{%
	\color{#1}\global\let\default@color\current@color
}
\makeatother

\DeclareSymbolFont{timesoperators}{T1}{ptm}{m}{n}
\SetSymbolFont{timesoperators}{bold}{T1}{ptm}{b}{n}

\makeatletter
\renewcommand{\operator@font}{\mathgroup\symtimesoperators}
\makeatother

\DeclareMathAlphabet{\mathbbm}{U}{bbm}{m}{n}
\DeclareFontFamily{U}{BOONDOX-calo}{\skewchar\font=45 }
\DeclareFontShape{U}{BOONDOX-calo}{m}{n}{
	<-> s*[1.05] BOONDOX-r-calo}{}
\DeclareFontShape{U}{BOONDOX-calo}{b}{n}{
	<-> s*[1.05] BOONDOX-b-calo}{}
\DeclareMathAlphabet{\mcb}{U}{BOONDOX-calo}{m}{n}
\SetMathAlphabet{\mcb}{bold}{U}{BOONDOX-calo}{b}{n}

\makeatletter 

\newcommand*{\fat}{}
\DeclareRobustCommand*{\fat}{%
	\mathbin{\mathpalette\bigcdot@{}}}
\newcommand*{\bigcdot@scalefactor}{.5}
\newcommand*{\bigcdot@widthfactor}{1.15}
\newcommand*{\bigcdot@}[2]{%
	\sbox0{$#1\vcenter{}$}
	\sbox2{$#1\cdot\m@th$}%
	\hbox to \bigcdot@widthfactor\wd2{%
		\hfil
		\raise\ht0\hbox{%
			\scalebox{\bigcdot@scalefactor}{%
				\lower\ht0\hbox{$#1\bullet\m@th$}%
			}%
		}%
		\hfil
	}%
}

\DeclareRobustCommand{\TitleEquation}[2]{\texorpdfstring{\StrLeft{\f@series}{1}[\@firstchar]$\if b\@firstchar\boldsymbol{#1}\else#1\fi$}{#2}}

\makeatother

\theoremnumbering{Alph}
\renewtheorem{theorem*}{Theorem}
\newtheorem{assumption}[lemma]{Assumption}
\newtheorem{example}[lemma]{Example}
\numberwithin{equation}{section}

\def\slash{\leavevmode\unskip\kern0.18em/\penalty\exhyphenpenalty\kern0.18em}
\def\dash{\leavevmode\unskip\kern0.18em--\penalty\exhyphenpenalty\kern0.18em}

\let\epsilon\varepsilon

\def\${|\!|\!|}

\setlist{noitemsep,topsep=4pt}
\def\para_#1{/\!\!/_{\!#1}}

%% file: bibliography.bib
@article {Li-Gao-Qu:25,
    AUTHOR = {Li, Haoyuan and Gao, Hongjun and Qu, Shiduo},
     TITLE = {Averaging principle for slow-fast {SPDE}s driven by mixed
              noises},
   JOURNAL = {J. Differential Equations},
  FJOURNAL = {Journal of Differential Equations},
    VOLUME = {430},
      YEAR = {2025},
     PAGES = {Paper No. 113209, 51},
      ISSN = {0022-0396,1090-2732},
   MRCLASS = {35R60 (35K90 60G22 60H05)},
  MRNUMBER = {4875482},
MRREVIEWER = {Mustafa\ Avci},
       DOI = {10.1016/j.jde.2025.02.080},
       URL = {https://doi.org/10.1016/j.jde.2025.02.080},
}

@article {Xu-Lian-Wu:23,
    AUTHOR = {Xu, Jie and Lian, Qiqi and Wu, Jiang-Lun},
     TITLE = {A strong averaging principle rate for two-time-scale coupled
              forward-backward stochastic differential equations driven by
              fractional {B}rownian motion},
   JOURNAL = {Appl. Math. Optim.},
  FJOURNAL = {Applied Mathematics and Optimization},
    VOLUME = {88},
      YEAR = {2023},
    NUMBER = {2},
     PAGES = {Paper No. 32, 35},
      ISSN = {0095-4616,1432-0606},
   MRCLASS = {60H10 (60H07 60H35 70K65 70K70)},
  MRNUMBER = {4595992},
MRREVIEWER = {Jan\ I.\ Seidler},
       DOI = {10.1007/s00245-023-10008-2},
       URL = {https://doi.org/10.1007/s00245-023-10008-2},
}

@article {Gu-Bal:12,
    AUTHOR = {Gu, Yu and Bal, Guillaume},
     TITLE = {Random homogenization and convergence to integrals with
              respect to the {R}osenblatt process},
   JOURNAL = {J. Differential Equations},
  FJOURNAL = {Journal of Differential Equations},
    VOLUME = {253},
      YEAR = {2012},
    NUMBER = {4},
     PAGES = {1069--1087},
      ISSN = {0022-0396,1090-2732},
   MRCLASS = {35B27 (35J25)},
  MRNUMBER = {2925905},
       DOI = {10.1016/j.jde.2012.05.007},
       URL = {https://doi.org/10.1016/j.jde.2012.05.007},
}

@Article{bianchi-Bonaccorsi-Canadas-Friesen,
author={Bianchi, Luigi Amedeo
and Bonaccorsi, Stefano
and Ca{\~{n}}adas, Ole
and Friesen, Martin},
title={Limit theorems for stochastic Volterra processes},
journal={Stochastics and Partial Differential Equations: Analysis and Computations},
year={2026},
month={May},
day={18},
issn={2194-041X},
doi={10.1007/s40072-026-00428-w},
url={https://doi.org/10.1007/s40072-026-00428-w}
}

@unpublished{Brehier-Faye:25,
      title={Averaging principle for slow-fast fractional stochastic differential equations}, 
      author={Charles-Edouard Bréhier and Ibrahima Faye},
      year={2025},
      note="arXiv:2510.04129",
      eprint={2510.04129},
      archivePrefix={arXiv},
      primaryClass={math.PR},
      url={https://arxiv.org/abs/2510.04129}, 
}

@article {Friz-Gassaiat-Lyons:15,
    AUTHOR = {Friz, Peter and Gassiat, Paul and Lyons, Terry},
     TITLE = {Physical {B}rownian motion in a magnetic field as a rough
              path},
   JOURNAL = {Trans. Amer. Math. Soc.},
  FJOURNAL = {Transactions of the American Mathematical Society},
    VOLUME = {367},
      YEAR = {2015},
    NUMBER = {11},
     PAGES = {7939--7955},
      ISSN = {0002-9947,1088-6850},
   MRCLASS = {60H99},
  MRNUMBER = {3391905},
MRREVIEWER = {Danyu\ Yang},
       DOI = {10.1090/S0002-9947-2015-06272-2},
       URL = {https://doi.org/10.1090/S0002-9947-2015-06272-2},
}

@article {Jaramillo-Nopurdin-Nualart-Peccati:23,
    AUTHOR = {Jaramillo, Arturo and Nourdin, Ivan and Nualart, David and
              Peccati, Giovanni},
     TITLE = {Limit theorems for additive functionals of the fractional
              {B}rownian motion},
   JOURNAL = {Ann. Probab.},
  FJOURNAL = {The Annals of Probability},
    VOLUME = {51},
      YEAR = {2023},
    NUMBER = {3},
     PAGES = {1066--1111},
      ISSN = {0091-1798,2168-894X},
   MRCLASS = {60G22 (60F05 60H07 60J55)},
  MRNUMBER = {4583063},
MRREVIEWER = {Tam\'as\ Szabados},
       DOI = {10.1214/22-aop1612},
       URL = {https://doi.org/10.1214/22-aop1612},
}

@article {Nourdin-Peccati-Reinert:10,
    AUTHOR = {Nourdin, Ivan and Peccati, Giovanni and Reinert, Gesine},
     TITLE = {Invariance principles for homogeneous sums: universality of
              {G}aussian {W}iener chaos},
   JOURNAL = {Ann. Probab.},
  FJOURNAL = {The Annals of Probability},
    VOLUME = {38},
      YEAR = {2010},
    NUMBER = {5},
     PAGES = {1947--1985},
      ISSN = {0091-1798,2168-894X},
   MRCLASS = {60F05 (60F17 60G15 60H07)},
  MRNUMBER = {2722791},
       DOI = {10.1214/10-AOP531},
       URL = {https://doi.org/10.1214/10-AOP531},
}

@article {Ye-Yang-Maggioni:24,
    AUTHOR = {Ye, Felix X.-F. and Yang, Sichen and Maggioni, Mauro},
     TITLE = {Nonlinear model reduction for slow-fast stochastic systems
              near unknown invariant manifolds},
   JOURNAL = {J. Nonlinear Sci.},
  FJOURNAL = {Journal of Nonlinear Science},
    VOLUME = {34},
      YEAR = {2024},
    NUMBER = {1},
     PAGES = {Paper No. 22, 54},
      ISSN = {0938-8974,1432-1467},
   MRCLASS = {60H35 (37H10)},
  MRNUMBER = {4675244},
MRREVIEWER = {Marco\ P.\ Cabral},
       DOI = {10.1007/s00332-023-09998-8},
       URL = {https://doi.org/10.1007/s00332-023-09998-8},
}

@article {Han-Xu-Pei-Wu,
    AUTHOR = {Han, Min and Xu, Yong and Pei, Bin and Wu, Jiang-Lun},
     TITLE = {Two-time-scale stochastic differential delay equations driven
              by multiplicative fractional {B}rownian noise: averaging
              principle},
   JOURNAL = {J. Math. Anal. Appl.},
  FJOURNAL = {Journal of Mathematical Analysis and Applications},
    VOLUME = {510},
      YEAR = {2022},
    NUMBER = {2},
     PAGES = {Paper No. 126004, 31},
      ISSN = {0022-247X,1096-0813},
   MRCLASS = {47H30 (47H05 60G22 60H10)},
  MRNUMBER = {4372789},
       DOI = {10.1016/j.jmaa.2022.126004},
       URL = {https://doi.org/10.1016/j.jmaa.2022.126004},
}

@article {Wang-Xu-Pei:23,
    AUTHOR = {Wang, Ruifang and Xu, Yong and Pei, Bin},
     TITLE = {Stochastic averaging for a completely integrable {H}amiltonian
              system with fractional {B}rownian motion},
   JOURNAL = {Stoch. Dyn.},
  FJOURNAL = {Stochastics and Dynamics},
    VOLUME = {23},
      YEAR = {2023},
    NUMBER = {4},
     PAGES = {Paper No. 2350026, 23},
      ISSN = {0219-4937,1793-6799},
   MRCLASS = {60G22 (34C29 37J35 60H10)},
  MRNUMBER = {4631639},
       DOI = {10.1142/S0219493723500260},
       URL = {https://doi.org/10.1142/S0219493723500260},
}

@article {Wang-Chao-Duan,
    AUTHOR = {Wei, Pingyuan and Chao, Ying and Duan, Jinqiao},
     TITLE = {Hamiltonian systems with {L}\'evy noise: symplecticity,
              {H}amilton's principle and averaging principle},
   JOURNAL = {Phys. D},
  FJOURNAL = {Physica D. Nonlinear Phenomena},
    VOLUME = {398},
      YEAR = {2019},
     PAGES = {69--83},
      ISSN = {0167-2789,1872-8022},
   MRCLASS = {37H10 (60H10 60J76)},
  MRNUMBER = {3969229},
       DOI = {10.1016/j.physd.2019.05.013},
       URL = {https://doi.org/10.1016/j.physd.2019.05.013},
}

@article {Pei-Inahama-Xu21,
    AUTHOR = {Pei, Bin and Inahama, Yuzuru and Xu, Yong},
     TITLE = {Averaging principle for fast-slow system driven by mixed
              fractional {B}rownian rough path},
   JOURNAL = {J. Differential Equations},
  FJOURNAL = {Journal of Differential Equations},
    VOLUME = {301},
      YEAR = {2021},
     PAGES = {202--235},
      ISSN = {0022-0396,1090-2732},
   MRCLASS = {60G22 (34C29 60H10 60L20)},
  MRNUMBER = {4305934},
       DOI = {10.1016/j.jde.2021.08.006},
       URL = {https://doi.org/10.1016/j.jde.2021.08.006},
}

@article {Blomker-Neamtu22,
    AUTHOR = {Bl\"omker, Dirk and Neam\c{t}u, Alexandra},
     TITLE = {Amplitude equations for {SPDE}s driven by fractional additive
              noise with small {H}urst parameter},
   JOURNAL = {Stoch. Dyn.},
  FJOURNAL = {Stochastics and Dynamics},
    VOLUME = {22},
      YEAR = {2022},
    NUMBER = {3},
     PAGES = {Paper No. 2240013, 33},
      ISSN = {0219-4937,1793-6799},
   MRCLASS = {60H15 (60G22 60H05)},
  MRNUMBER = {4439594},
MRREVIEWER = {Cheuk\ Yin\ Lee},
       DOI = {10.1142/S0219493722400135},
       URL = {https://doi.org/10.1142/S0219493722400135},
}

@incollection {Li:cons,
    AUTHOR = {Li, Xue-Mei},
     TITLE = {Perturbation of conservation laws and averaging on manifolds},
 BOOKTITLE = {Computation and combinatorics in dynamics, stochastics and
              control},
    SERIES = {Abel Symp.},
    VOLUME = {13},
     PAGES = {499--550},
 PUBLISHER = {Springer, Cham},
      YEAR = {2018},
      ISBN = {978-3-030-01592-3; 978-3-030-01593-0},
   MRCLASS = {60H10 (34C29 34F05 58J65)},
  MRNUMBER = {3967395},
}

@incollection {Gehringer-Li21,
    AUTHOR = {Gehringer, Johann and Li, Xue-Mei},
     TITLE = {Rough homogenisation with fractional dynamics},
 BOOKTITLE = {Geometry and invariance in stochastic dynamics},
    SERIES = {Springer Proc. Math. Stat.},
    VOLUME = {378},
     PAGES = {137--168},
 PUBLISHER = {Springer, Cham},
      YEAR = {[2021] \copyright 2021},
      ISBN = {978-3-030-87431-5; 978-3-030-87432-2},
   MRCLASS = {60-02 (60F05 60G22 60H10 60L20)},
  MRNUMBER = {4432538},
       DOI = {10.1007/978-3-030-87432-2\_8},
       URL = {https://doi.org/10.1007/978-3-030-87432-2_8},
}

@article{Gehringer-Li-Sieber:22,
   title={{Functional Limit Theorems for {V}olterra Processes and Applications to Homogenization}},
   volume={35},
   ISSN={1361-6544},
   url={http://dx.doi.org/10.1088/1361-6544/ac4818},
   DOI={10.1088/1361-6544/ac4818},
   number={4},
   journal={Nonlinearity},
   publisher={IOP Publishing},
   author={Gehringer, Johann and Li, Xue-Mei and Sieber, Julian},
   year={2022},
   month=mar, pages={1521--1557} }

@article {Gehringer-Li22,
    AUTHOR = {Gehringer, Johann and Li, Xue-Mei},
     TITLE = {Functional limit theorems for the fractional
              {O}rnstein-{U}hlenbeck process},
   JOURNAL = {J. Theoret. Probab.},
  FJOURNAL = {Journal of Theoretical Probability},
    VOLUME = {35},
      YEAR = {2022},
    NUMBER = {1},
     PAGES = {426--456},
      ISSN = {0894-9840,1572-9230},
   MRCLASS = {60F17 (60G22)},
  MRNUMBER = {4379471},
       DOI = {10.1007/s10959-020-01044-7},
       URL = {https://doi.org/10.1007/s10959-020-01044-7},
}

@article {Li:ODE,
    AUTHOR = {Li, Xue-Mei},
     TITLE = {Limits of random differential equations on manifolds},
   JOURNAL = {Probab. Theory Related Fields},
  FJOURNAL = {Probability Theory and Related Fields},
    VOLUME = {166},
      YEAR = {2016},
    NUMBER = {3-4},
     PAGES = {659--712},
      ISSN = {0178-8051,1432-2064},
   MRCLASS = {60H25 (34F05 58J65 60B10 60F05 60H10 60J60)},
  MRNUMBER = {3568037},
       DOI = {10.1007/s00440-015-0669-x},
       URL = {https://doi.org/10.1007/s00440-015-0669-x},
}

@article {Le:20,
    AUTHOR = {L\^e, Khoa},
     TITLE = {A stochastic sewing lemma and applications},
   JOURNAL = {Electron. J. Probab.},
  FJOURNAL = {Electronic Journal of Probability},
    VOLUME = {25},
      YEAR = {2020},
     PAGES = {1--55},
       DOI = {10.1214/20-EJP442},
       URL = {https://doi.org/10.1214/20-EJP442},
}

@article {Roeckner-Xie21,
    AUTHOR = {R\"ockner, Michael and Xie, Longjie},
     TITLE = {Averaging principle and normal deviations for multiscale
              stochastic systems},
   JOURNAL = {Comm. Math. Phys.},
  FJOURNAL = {Communications in Mathematical Physics},
    VOLUME = {383},
      YEAR = {2021},
    NUMBER = {3},
     PAGES = {1889--1937},
      ISSN = {0010-3616,1432-0916},
   MRCLASS = {60H10 (60F17 60H15)},
  MRNUMBER = {4244263},
       DOI = {10.1007/s00220-021-04069-z},
       URL = {https://doi.org/10.1007/s00220-021-04069-z},
}

@article {Eichinger-Kuehn-Neamu20,
    AUTHOR = {Eichinger, Katharina and Kuehn, Christian and Neam\c{t}u,
              Alexandra},
     TITLE = {Sample paths estimates for stochastic fast-slow systems driven
              by fractional {B}rownian motion},
   JOURNAL = {J. Stat. Phys.},
  FJOURNAL = {Journal of Statistical Physics},
    VOLUME = {179},
      YEAR = {2020},
    NUMBER = {5-6},
     PAGES = {1222--1266},
      ISSN = {0022-4715,1572-9613},
   MRCLASS = {34E13 (60G22 60H10 86A05)},
  MRNUMBER = {4123355},
MRREVIEWER = {Mireia\ Besal\'u},
       DOI = {10.1007/s10955-020-02485-4},
       URL = {https://doi.org/10.1007/s10955-020-02485-4},
}

@article {Solesne-Siragan-Konstantinos21,
    AUTHOR = {Bourguin, Solesne and Gailus, Siragan and Spiliopoulos,
              Konstantinos},
     TITLE = {Discrete-time inference for slow-fast systems driven by
              fractional {B}rownian motion},
   JOURNAL = {Multiscale Model. Simul.},
  FJOURNAL = {Multiscale Modeling \& Simulation. A SIAM Interdisciplinary
              Journal},
    VOLUME = {19},
      YEAR = {2021},
    NUMBER = {3},
     PAGES = {1333--1366},
      ISSN = {1540-3459,1540-3467},
   MRCLASS = {60G22 (60H07 60H10 62F12)},
  MRNUMBER = {4301917},
       DOI = {10.1137/20M135813X},
       URL = {https://doi.org/10.1137/20M135813X},
}

@article {Gailus-Gasteratos25,
    AUTHOR = {Gailus, Siragan and Gasteratos, Ioannis},
     TITLE = {Large deviations of slow-fast systems driven by fractional
              {B}rownian motion},
   JOURNAL = {Electron. J. Probab.},
  FJOURNAL = {Electronic Journal of Probability},
    VOLUME = {30},
      YEAR = {2025},
     PAGES = {Paper No. 80, 56},
      ISSN = {1083-6489},
   MRCLASS = {60H10 (60F10 60G22)},
  MRNUMBER = {4904101},
       DOI = {10.1214/25-ejp1340},
       URL = {https://doi.org/10.1214/25-ejp1340},
}

@article {Solesne-Dang-Spiliopoulos24,
    AUTHOR = {Bourguin, Solesne and Dang, Thanh and Spiliopoulos,
              Konstantinos},
     TITLE = {Moderate deviation principle for multiscale systems driven by
              fractional {B}rownian motion},
   JOURNAL = {J. Theoret. Probab.},
  FJOURNAL = {Journal of Theoretical Probability},
    VOLUME = {37},
      YEAR = {2024},
    NUMBER = {1},
     PAGES = {352--408},
      ISSN = {0894-9840,1572-9230},
   MRCLASS = {60H10 (60F10 60G22 60H07)},
  MRNUMBER = {4715719},
MRREVIEWER = {Cheuk\ Yin\ Lee},
       DOI = {10.1007/s10959-023-01235-y},
       URL = {https://doi.org/10.1007/s10959-023-01235-y},
}

@article {Chekroun-Liu-McWilliams25,
    AUTHOR = {Chekroun, Micka\"el D. and Liu, Honghu and McWilliams, James
              C.},
     TITLE = {Non-{M}arkovian reduced models to unravel transitions in
              non-equilibrium systems},
   JOURNAL = {J. Phys. A},
  FJOURNAL = {Journal of Physics. A. Mathematical and Theoretical},
    VOLUME = {58},
      YEAR = {2025},
    NUMBER = {4},
     PAGES = {Paper No. 045204, 60},
      ISSN = {1751-8113,1751-8121},
   MRCLASS = {82C26 (37L55)},
  MRNUMBER = {4867452},
}

@article {Friz-Salkeld-Wagenhofer25,
    AUTHOR = {Friz, Peter K. and Salkeld, William and Wagenhofer, Thomas},
     TITLE = {Weak error estimates for rough volatility models},
   JOURNAL = {Ann. Appl. Probab.},
  FJOURNAL = {The Annals of Applied Probability},
    VOLUME = {35},
      YEAR = {2025},
    NUMBER = {1},
     PAGES = {64--98},
      ISSN = {1050-5164,2168-8737},
   MRCLASS = {60G22 (60L90 91B70)},
  MRNUMBER = {4871704},
       DOI = {10.1214/24-aap2109},
       URL = {https://doi.org/10.1214/24-aap2109},
}

@article {Berglund-Blessing25,
    AUTHOR = {Berglund, Nils and Blessing, Alexandra},
     TITLE = {Concentration estimates for {SPDE}s driven by fractional
              {B}rownian motion},
   JOURNAL = {Electron. Commun. Probab.},
  FJOURNAL = {Electronic Communications in Probability},
    VOLUME = {30},
      YEAR = {2025},
     PAGES = {Paper No. 20, 13},
      ISSN = {1083-589X},
   MRCLASS = {60H15 (60G17 60G22)},
  MRNUMBER = {4873043},
       DOI = {10.1214/25-ECP664},
       URL = {https://doi.org/10.1214/25-ECP664},
}

@article {Alonso-Martin-Boedihardjo-Papavasiliou25,
    AUTHOR = {Alonso-Martin, Pablo Ramses and Boedihardjo, Horatio and
              Papavasiliou, Anastasia},
     TITLE = {Statistical inference for the rough homogenization limit of
              multiscale fractional {O}rnstein-{U}hlenbeck processes},
   JOURNAL = {Electron. Commun. Probab.},
  FJOURNAL = {Electronic Communications in Probability},
    VOLUME = {30},
      YEAR = {2025},
     PAGES = {Paper No. 29, 13},
      ISSN = {1083-589X},
   MRCLASS = {60G22 (60H30 62F12 62M09)},
  MRNUMBER = {4892286},
       DOI = {10.1214/25-ecp677},
       URL = {https://doi.org/10.1214/25-ecp677},
}

@article {Djurdjevac-Kremp-Perkowski25,
    AUTHOR = {Djurdjevac, Ana and Kremp, Helena and Perkowski, Nicolas},
     TITLE = {Rough homogenization for {L}angevin dynamics on fluctuating
              {H}elfrich surfaces},
   JOURNAL = {Stoch. Anal. Appl.},
  FJOURNAL = {Stochastic Analysis and Applications},
    VOLUME = {43},
      YEAR = {2025},
    NUMBER = {3},
     PAGES = {423--445},
      ISSN = {0736-2994,1532-9356},
   MRCLASS = {60L20 (35B27 60F05 60H30 60J55)},
  MRNUMBER = {4920370},
       DOI = {10.1080/07362994.2025.2505736},
       URL = {https://doi.org/10.1080/07362994.2025.2505736},
}

@article {Roeckner-Xie-Yang:23,
    AUTHOR = {R\"ockner, Michael and Xie, Longjie and Yang, Li},
     TITLE = {Asymptotic behavior of multiscale stochastic partial
              differential equations with {H}\"older coefficients},
   JOURNAL = {J. Funct. Anal.},
  FJOURNAL = {Journal of Functional Analysis},
    VOLUME = {285},
      YEAR = {2023},
    NUMBER = {9},
     PAGES = {Paper No. 110103, 50},
      ISSN = {0022-1236,1096-0783},
   MRCLASS = {60H15 (60F05 70K65)},
  MRNUMBER = {4625953},
MRREVIEWER = {Markus\ Kunze},
       DOI = {10.1016/j.jfa.2023.110103},
       URL = {https://doi.org/10.1016/j.jfa.2023.110103},
}

@article {Li:16,
    AUTHOR = {Li, Xue-Mei},
     TITLE = {Random perturbation to the geodesic equation},
   JOURNAL = {Ann. Probab.},
  FJOURNAL = {The Annals of Probability},
    VOLUME = {44},
      YEAR = {2016},
    NUMBER = {1},
     PAGES = {544--566},
      ISSN = {0091-1798,2168-894X},
   MRCLASS = {60H10 (37Hxx 53B05 58J65)},
  MRNUMBER = {3456345},
MRREVIEWER = {Dejun\ Luo},
       DOI = {10.1214/14-AOP981},
       URL = {https://doi.org/10.1214/14-AOP981},
}

@article {Li:08,
    AUTHOR = {Li, Xue-Mei},
     TITLE = {An averaging principle for a completely integrable stochastic
              {H}amiltonian system},
   JOURNAL = {Nonlinearity},
  FJOURNAL = {Nonlinearity},
    VOLUME = {21},
      YEAR = {2008},
    NUMBER = {4},
     PAGES = {803--822},
      ISSN = {0951-7715,1361-6544},
   MRCLASS = {60H10 (34C29 34F05 37H10 37J35 58J65)},
  MRNUMBER = {2399826},
MRREVIEWER = {Kiyomasa\ Narita},
       DOI = {10.1088/0951-7715/21/4/008},
       URL = {https://doi.org/10.1088/0951-7715/21/4/008},
}

@article {Gautam-Komorowski-Novikov-Rhyzhik14,
    AUTHOR = {Iyer, Gautam and Komorowski, Tomasz and Novikov, Alexei and
              Ryzhik, Lenya},
     TITLE = {From homogenization to averaging in cellular flows},
   JOURNAL = {Ann. Inst. H. Poincar\'e{} C Anal. Non Lin\'eaire},
  FJOURNAL = {Annales de l'Institut Henri Poincar\'e{} C. Analyse Non
              Lin\'eaire},
    VOLUME = {31},
      YEAR = {2014},
    NUMBER = {5},
     PAGES = {957--983},
      ISSN = {0294-1449,1873-1430},
   MRCLASS = {35B27 (35B45 35J25)},
  MRNUMBER = {3258362},
MRREVIEWER = {Claudia\ M.\ Timofte},
       DOI = {10.1016/j.anihpc.2013.06.003},
       URL = {https://doi.org/10.1016/j.anihpc.2013.06.003},
}

@book {Kmorowski-Landim-Olla12,
    AUTHOR = {Komorowski, Tomasz and Landim, Claudio and Olla, Stefano},
     TITLE = {Fluctuations in {M}arkov processes},
    SERIES = {Grundlehren der mathematischen Wissenschaften [Fundamental
              Principles of Mathematical Sciences]},
    VOLUME = {345},
      NOTE = {Time symmetry and martingale approximation},
 PUBLISHER = {Springer, Heidelberg},
      YEAR = {2012},
     PAGES = {xviii+491},
      ISBN = {978-3-642-29879-0},
   MRCLASS = {60J25 (60F05 60G44 60J10 60J60 60K35 60K37)},
  MRNUMBER = {2952852},
MRREVIEWER = {B\'alint\ T\'oth},
       DOI = {10.1007/978-3-642-29880-6},
       URL = {https://doi.org/10.1007/978-3-642-29880-6},
}

@incollection {Landim03,
    AUTHOR = {Landim, Claudio},
     TITLE = {Central limit theorem for {M}arkov processes},
 BOOKTITLE = {From classical to modern probability},
    SERIES = {Progr. Probab.},
    VOLUME = {54},
     PAGES = {145--205},
 PUBLISHER = {Birkh\"auser, Basel},
      YEAR = {2003},
      ISBN = {3-7643-2169-5},
   MRCLASS = {60F05 (60J05 60J25 60J55)},
  MRNUMBER = {2045987},
MRREVIEWER = {Ren\'e\ L.\ Schilling},
}

@book {Freidlin-Wentzell,
    AUTHOR = {Freidlin, M. I. and Wentzell, A. D.},
     TITLE = {Random perturbations of dynamical systems},
    SERIES = {Grundlehren der mathematischen Wissenschaften [Fundamental
              Principles of Mathematical Sciences]},
    VOLUME = {260},
   EDITION = {Second},
      NOTE = {Translated from the 1979 Russian original by Joseph Sz\"ucs},
 PUBLISHER = {Springer-Verlag, New York},
      YEAR = {1998},
     PAGES = {xii+430},
      ISBN = {0-387-98362-7},
   MRCLASS = {60H25 (34C29 34C35 34F05 58F30 60F10)},
  MRNUMBER = {1652127},
       DOI = {10.1007/978-1-4612-0611-8},
       URL = {https://doi.org/10.1007/978-1-4612-0611-8},
}

@article {Donsker51,
    AUTHOR = {Donsker, Monroe D.},
     TITLE = {An invariance principle for certain probability limit
              theorems},
   JOURNAL = {Mem. Amer. Math. Soc.},
  FJOURNAL = {Memoirs of the American Mathematical Society},
    VOLUME = {6},
      YEAR = {1951},
     PAGES = {12},
      ISSN = {0065-9266,1947-6221},
   MRCLASS = {60.0X},
  MRNUMBER = {40613},
MRREVIEWER = {J.\ L.\ Doob},
}

@article {Helland82,
    AUTHOR = {Helland, Inge S.},
     TITLE = {Central limit theorems for martingales with discrete or
              continuous time},
   JOURNAL = {Scand. J. Statist.},
  FJOURNAL = {Scandinavian Journal of Statistics. Theory and Applications},
    VOLUME = {9},
      YEAR = {1982},
    NUMBER = {2},
     PAGES = {79--94},
      ISSN = {0303-6898},
   MRCLASS = {60F05 (60G42 60G44)},
  MRNUMBER = {668684},
MRREVIEWER = {Holger\ Rootz\'en},
}

@incollection {Masi-Ferrari-Goldstein-Wick84,
    AUTHOR = {De Masi, A. and Ferrari, P. A. and Goldstein, S. and Wick, W.
              D.},
     TITLE = {Invariance principle for reversible {M}arkov processes with
              application to diffusion in the percolation regime},
 BOOKTITLE = {Particle systems, random media and large deviations
              ({B}runswick, {M}aine, 1984)},
    SERIES = {Contemp. Math.},
    VOLUME = {41},
     PAGES = {71--85},
 PUBLISHER = {Amer. Math. Soc., Providence, RI},
      YEAR = {1985},
      ISBN = {0-8218-5042-3},
   MRCLASS = {60J15 (60K35 82A43)},
  MRNUMBER = {814703},
MRREVIEWER = {C.\ van den Broeck},
       DOI = {10.1090/conm/041/814703},
       URL = {https://doi.org/10.1090/conm/041/814703},
}

@article {Kipnis-Varadhan86,
    AUTHOR = {Kipnis, C. and Varadhan, S. R. S.},
     TITLE = {Central limit theorem for additive functionals of reversible
              {M}arkov processes and applications to simple exclusions},
   JOURNAL = {Comm. Math. Phys.},
  FJOURNAL = {Communications in Mathematical Physics},
    VOLUME = {104},
      YEAR = {1986},
    NUMBER = {1},
     PAGES = {1--19},
      ISSN = {0010-3616,1432-0916},
   MRCLASS = {60F17 (60J05 60K35)},
  MRNUMBER = {834478},
MRREVIEWER = {H.\ Kesten},
       URL = {http://projecteuclid.org/euclid.cmp/1104114929},
}

@article{Hairer-Li:22,
   title={{Generating Diffusions with Fractional {B}rownian Motion}},
   volume={396},
   ISSN={1432-0916},
   url={http://dx.doi.org/10.1007/s00220-022-04462-2},
   DOI={10.1007/s00220-022-04462-2},
   number={1},
   journal={Communications in Mathematical Physics},
   publisher={Springer Science and Business Media LLC},
   author={Hairer, Martin and Li, Xue-Mei},
   year={2022},
   month=aug, pages={91--141} }

@article{Hairer-Li:20,
   title={{Averaging Dynamics Driven by Fractional {B}rownian Motion}},
   volume={48},
   ISSN={0091-1798},
   url={http://dx.doi.org/10.1214/19-AOP1408},
   DOI={10.1214/19-aop1408},
   number={4},
   journal={The Annals of Probability},
   publisher={Institute of Mathematical Statistics},
   author={Hairer, Martin and Li, Xue-Mei},
   year={2020},
   month=jul }

@inproceedings {HarrisOrig,
    AUTHOR = {Harris, T. E.},
     TITLE = {The existence of stationary measures for certain {M}arkov
              processes},
 BOOKTITLE = {Proceedings of the {T}hird {B}erkeley {S}ymposium on
              {M}athematical {S}tatistics and {P}robability, 1954--1955,
              vol. {II}},
     PAGES = {113--124},
 PUBLISHER = {University of California Press, Berkeley and Los Angeles,
              Calif.},
      YEAR = {1956},
   MRCLASS = {60.0X},
  MRNUMBER = {0084889},
MRREVIEWER = {G. A. Hunt},
}

@incollection {Harris,
    AUTHOR = {Hairer, Martin and Mattingly, Jonathan C.},
     TITLE = {Yet another look at {H}arris' ergodic theorem for {M}arkov
              chains},
 BOOKTITLE = {Seminar on {S}tochastic {A}nalysis, {R}andom {F}ields and
              {A}pplications {VI}},
    SERIES = {Progr. Probab.},
    VOLUME = {63},
     PAGES = {109--117},
     eprint={0810.2777},
 PUBLISHER = {Birkh\"{a}user/Springer Basel AG, Basel},
      YEAR = {2011},
   MRCLASS = {60J05 (37A30 37A50 47D07)},
  MRNUMBER = {2857021},
MRREVIEWER = {Wojciech Bartoszek},
       DOI = {10.1007/978-3-0348-0021-1_7},
       URL = {https://doi.org/10.1007/978-3-0348-0021-1_7},
}

@article {HMS,
    AUTHOR = {Hairer, M. and Mattingly, J. C. and Scheutzow, M.},
     TITLE = {Asymptotic coupling and a general form of {H}arris' theorem
              with applications to stochastic delay equations},
   JOURNAL = {Probab. Theory Related Fields},
  FJOURNAL = {Probability Theory and Related Fields},
    VOLUME = {149},
      YEAR = {2011},
    NUMBER = {1-2},
     PAGES = {223--259},
     eprint={0902.4495},
      ISSN = {0178-8051},
   MRCLASS = {60J05 (34K50 37A50 47N30 60H10)},
  MRNUMBER = {2773030},
MRREVIEWER = {Ramon van Handel},
       DOI = {10.1007/s00440-009-0250-6},
       URL = {https://doi.org/10.1007/s00440-009-0250-6},
}

@misc{Li-Ying:25,
      title={{Strong Completeness of {SDE}s and Non-explosion for {RDE}s with Coefficients Having Unbounded Derivatives}},
      author={Xue-Mei Li and Kexing Ying},
      year={2025},
      note={To appear in {\em Ann. Probab.}},
      eprint={2502.08799},
      archivePrefix={arXiv},
      primaryClass={math.PR},
      url={https://arxiv.org/abs/2502.08799}, 
}

@unpublished{arxiv,
      title={Fluctuations from a random fractional averaging limit},
      author={Li, Xue-Mei and Piernot, Colin and Sobczak, Szymon and Ying, Kexing},
      year={2025},
      note={arXiv:2512.08621v1},
      eprint={2512.08621v1},
      archivePrefix={arXiv},
      primaryClass={math.PR},
      url={https://arxiv.org/abs/2512.08621v1},
}

@article{Catellier-Gubinelli:16,
      author={Catellier, R\'emi and Gubinelli, Massimiliano},
      title={Averaging along irregular curves and regularisation of {ODE}s},
      journal={Stochastic Process. Appl.},
      volume={126},
      number={8},
      pages={2323--2366},
      year={2016},
}

@article{Hu-Le:17,
      author={Hu, Yaozhong and L\^e, Khoa},
      title={Nonlinear {Y}oung integrals and differential systems in {H}\"older media},
      journal={Trans. Amer. Math. Soc.},
      volume={369},
      number={3},
      pages={1935--2002},
      year={2017},
}

@article{Galeati:23,
      author={Galeati, Lucio},
      title={Nonlinear {Y}oung differential equations: a review},
      journal={J. Dynam. Differential Equations},
      volume={35},
      number={2},
      pages={985--1046},
      year={2023},
}

@article{Lejay:10,
      author={Lejay, Antoine},
      title={Controlled differential equations as {Y}oung integrals: a simple approach},
      journal={J. Differential Equations},
      volume={249},
      number={8},
      pages={1777--1798},
      year={2010},
}

@article{Nualart-Xia:20,
      author={Nualart, David and Xia, Panqiu},
      title={{On Nonlinear Rough Paths}},
      journal={Lat. Am. J. Probab. Math. Stat.},
      volume={17},
      pages={545--587},
      year={2020},
}

@book{Lieb-Loss:01,
      author={Lieb, Elliott H. and Loss, Michael},
      title={Analysis},
      series={Graduate Studies in Mathematics},
      volume={14},
      edition={Second},
      publisher={American Mathematical Society},
      address={Providence, RI},
      year={2001},
}

@book{Billingsley:95,
  author = {Patrick Billingsley},
  publisher = {John Wiley \& Sons, Inc.},
  title = {Probability and Measure, 3rd edition},
  year = {1995}
}
